\documentclass[a4]{article}
\usepackage{amsmath,amsthm,amssymb,graphics}
\usepackage[latin1]{inputenc}
\usepackage{epsfig}
\usepackage[all]{xy}
\usepackage{tikz}
\usepackage{amsfonts}
\usepackage{pbox}
\usepackage{multicol}

\newtheorem{remark}{Remark}[section]
\newtheorem{example}{Example}[section]
\newtheorem{theorem}{Theorem}[section]
\newtheorem{proposition}{Proposition}[section]
\newtheorem{corollary}{Corollary}[section]

\begin{document}

\title{\bf On a four-parameters generalization of some special sequences}

\author{Robson da silva\footnote{Supported by CNPq}, Kelvin S. de Oliveira, and Almir C. G.  Neto}

\date{}
\maketitle

\begin{abstract}
We introduce a new four-parameters sequence that simultaneously generalizes some well-known integer sequences, including Fibonacci, Padovan, Jacobsthatl, Pell, and Lucas numbers. Combinatorial interpretations are discussed and many identities for this general sequence are derived. As a consequence, a number of identities for Fibonacci, Lucas, Pell, Jacobsthal, Padovan, and Narayana numbers as well as some of their generalizations are obtained. We also present the Cassini formula for the new sequence.
\end{abstract}

\noindent {\bf keyword}: Generalized Fibonacci numbers, Generalized Lucas numbers, tilings, identities.


\section{Introduction}

The Fibonacci numbers are defined recursively by $F_n = F_{n-1}+F_{n-2}$ for $n \geq 2$, with initial conditions $F_0=0$ and $F_1=1$. The Fibonacci numbers have been generalized by a number of authors in many different ways. In \cite{falcon2007fibonacci}, the authors introduced the $k$-Fibonacci numbers $F(k,n)=kF(k,n-1)+F(k,n-2)$, for $n\geq 2$, $F(k,0)=0$ and $F(k,1)=1$. In \cite{distance2}, the $(2,k)$-Fibonacci numbers were introduced by the recurrence relation $F_{2}(k,n)=F_{2}(k,n-2)+F_{2}(k,n-k)$, for $n\geq k$, with $F_{2}(k,n)=1$ for $n=0, \ldots, k-1$, and the authors interpreted them in terms of certain set decompositions. A two-variables polynomial generalization is presented in \cite{amdeberhan2014generalized}. Among others, some generalizations can be seen in \cite{bednarz2015three, distance, benjamin2008tiling, falcon2014generalized,  flaut2013generalized, K1}. Some of the generalizations have interesting and useful graph interpretations, see \cite{W1, W2, wloch2013some}.

In this paper, we introduce a new generalization not only for the Fibonacci numbers, but for a many other special integer sequences and their generalizations. Our four-parameters sequence, $F^{i}_{r,s}(k,n)$, is defined by the recurrence relation:
$$F^{i}_{r,s}(k,n) =  \begin{array}{l}
rF^{i}_{r,s}(k,n-i)+sF^{i}_{r,s}(k,n-k),
\end{array}$$
with appropriate initial conditions. We also introduce a new generalization of the Lucas numbers $L_{n}$:
$$L^{i}_{r,s}(k,n)=r^{k-1}\lbrace(k-1)sF^{i}_{r,s}(k,n-k-(k-1)i)+F^{i}_{r,s}(k,n-(k-1)i)\rbrace.$$

For $i=r=s=1$ and $k=2$ in $F^{i}_{r,s}(k,n)$ and $L^{i}_{r,s}(k,n)$, we obtain the numbers $F_{n}$ and $L_{n}$, respectively. In addition, Pell, Jacobsthal, Padovan and Narayana numbers  and some of their generalizations are obtained from $F^{i}_{r,s}(k,n)$ (Section \ref{Sectn1}).


We present three combinatorial interpretations for $F^{i}_{r,s}(k,n)$: one in terms of tilings, another in terms of color decomposition of sets and, finally, a graph interpretation. A tiling interpretation is also given for  $L^{i}_{r,s}(k,n)$.

We use the tiling counting technique, to obtain a closed formula for $F^{i}_{r,s}(k,n)$ and new identities for these sequences. The combinatorial aspects of counting via tilings has been explored and employed by many authors, see \cite{benjamin2008tiling, benjamin2003, briggs2010combinatorial, robson} for instance. Although we could have just used induction in the proofs, we preferred the tilling approach in order to make it clear where the identities stem from.



In Section \ref{Sec6} we are describe the matrix from where we can obtain the numbers $F_{r,s}^{i}(k,n)$. In other words, we develop to $F_{r,s}^{i}(k,n)$ the matrix approach that has been used by many authors for studying particular sequences, see for instance \cite{brualdi1991matrix, distance2}.

\section{The sequences $F_{r,s}^{i}(k,n)$ and $L_{r,s}^{i}(k,n)$}
\label{Sectn1}

In this section, we introduce the numbers $F_{r,s}^{i}(k,n)$ and $L_{r,s}^{i}(k,n)$ and provide examples of sequences that can be obtained from them by specializing the parameters $r,s,i$, and $k$. In particular, these new numbers are generalizations of the Fibonacci and Lucas numbers.

Let $i,k,r,s\geq 1$, and $n\geq-1$ be integers. We define the sequences $F^{i}_{r,s}(k,n)$ and $L^{i}_{r,s}(k,n)$ by the following recurrence relations:
\begin{equation}F^{i}_{r,s}(k,n) = \left\{ \begin{array}{l}
rF^{i}_{r,s}(k,n-i)+sF^{i}_{r,s}(k,n-k), \mbox{ for }  n \geq \max \lbrace k,i\rbrace-1 \\
r^{\lfloor\frac{n+1}{i}\rfloor}$, for $i \leq k$ and  $n=-1,0,1,\ldots,  k - 2 \\
s^{\lfloor\frac{n+1}{k}\rfloor}$, for $k < i$ and  $n=-1,0,1,\ldots ,i-2 \\
\end{array} \right.
\label{23.1}
\end{equation}
and
$$L^{i}_{r,s}(k,n)=r^{k-1}\lbrace(k-1)sF^{i}_{r,s}(k,n-k-(k-1)i)+F^{i}_{r,s}(k,n-(k-1)i)\rbrace,$$ for $k>i$ and $n \geq k+(k-1)i-1$, with initial conditions:
$$L^{i}_{r,s}(k,n)= \left\{ \begin{array}{l} 
F^{i}_{r,s}(k,n), \mbox{ for } n=-1,0, \ldots, k-2 \\
F^{i}_{r,s}(k,n)-r^{\lfloor \frac{n+1}{i}\rfloor}, \mbox{ for } n=k-1, \ldots, k+(k-1)i-2.
\end{array} \right.$$

The Fibonacci sequence can be obtained directly from $F^{i}_{r,s}(k,n)$ in two different ways: either by taking $r=s=i=1$ and $k=2$ by taking $r=s=k=1$ and $i=2$, i.e., $F^{1}_{1,1}(2,n)=F^{2}_{1,1}(1,n)=F_{n+2}$ for $n\geq-1$. We also note that $L^{1}_{1,1}(2,n)=L_n$, for $n\geq2$, the $n$th Lucas number. Table \ref{tab1} exhibits some other well-known sequences that can be derived from $F^{i}_{r,s}(k,n)$ and $L^{i}_{r,s}(k,n)$.

\begin{table}[h!]
	\scriptsize
	\centering
	\begin{tabular}{|c|c|c|c|l|l|l|}
		\hline
		$i$ & $k$ & $r$ & $s$ & \textit{Recurrences for} $F^{i}_{r,s}(k,n)$ \textit{or} $L^{i}_{r,s}(k,n)$ & \textit{Initial conditions} & \textit{References}\\ \hline \hline
		1 & 2 & 1 & 1 & \pbox{20cm}{$F^{1}_{1,1}(2,n) =$ \\ $F^{1}_{1,1}(2,n-1)+F^{1}_{1,1}(2,n-2), n \geq 1$} & \pbox{20cm}{$F^{1}_{1,1}(2,-1)=1$ \\ $ F^{1}_{1,1}(2,0)=1$}  & \pbox{20cm}{$F_n$, \textit{Fibonacci} \\ {oeis.org/A000045}} \\ \hline
		1 & 2 & 1 & 1 & \pbox{20cm}{$L^{1}_{1,1}(2,n) =$ \\ $L^{1}_{1,1}(2,n-1)+L^{1}_{1,1}(2,n-2), n \geq 3$} & \pbox{20cm}{$L^{1}_{1,1}(2,1)=1$ \\ $ L^{1}_{1,1}(2,2)=3$}  & \pbox{20cm}{$L_n$, \textit{Lucas} \\ {oeis.org/A000032}} \\ \hline
		1 & 2 & 2 & 1 & \pbox{20cm}{$F^{1}_{2,1}(2,n) =$ \\ $2F^{1}_{2,1}(2,n-1)+F^{1}_{2,1}(2,n-2), n \geq 1$} & \pbox{20cm}{$F^{1}_{2,1}(2,-1)=1$ \\ $ F^{1}_{2,1}(2,0)=2$}  & \pbox{20cm}{$P_n$, \textit{Pell} \\ {oeis.org/A000129}} \\ \hline
		1 & 2 & 1 & 2 & \pbox{20cm}{$F^{1}_{1,2}(2,n) =$ \\ $F^{1}_{1,2}(2,n-1)+ 2F^{1}_{1,2}(2,n-2), n \geq 1$} & \pbox{20cm}{$F^{1}_{1,2}(2,-1)=1$ \\ $ F^{1}_{1,2}(2,0)=1$} & \pbox{20cm}{$J_n$, \textit{Jacobsthal} \\ {oeis.org/A001045}} \\ \hline
		2 & 3 & 1 & 1 & \pbox{20cm}{$F^{2}_{1,1}(3,n) =$ \\ $F^{2}_{1,1}(3,n-2)+F^{2}_{1,1}(3,n-3), n \geq 2$} & \pbox{20cm}{$F^{2}_{1,1}(3,n)=1$ \\ $n=-1, 0, 1$}  & \pbox{20cm}{$Pv(n)$, \textit{Padovan} \\ {oeis.org/A000931}} \\ \hline
		1 & 3 & 1 & 1 & \pbox{20cm}{$F^{1}_{1,1}(3,n) =$ \\ $F^{1}_{1,1}(3,n-1)+F^{1}_{1,1}(3,n-3), n \geq 2$} & \pbox{20cm}{$F^{1}_{1,1}(3,n)=1$ \\ $n=-1, 0, 1$}  & \pbox{20cm}{$u_n$, \textit{Narayana} \\ {oeis.org/A000930}, \\ \cite{flaut2013generalized}} \\ \hline
		1 & 4 & 1 & 1 & \pbox{20cm}{$F^{1}_{1,1}(4,n) =$ \\ $F^{1}_{1,1}(4,n-1)+F^{1}_{1,1}(4,n-4), n \geq 3$} & \pbox{20cm}{$F^{1}_{1,1}(4,n)=1$ \\ $n=-1,0,1,2$}  & {oeis.org/A003269} \\ \hline
		1 & 2 & $a$ & $b$ & \pbox{20cm}{$F^{1}_{a,b}(2,n) =$ \\ $aF^{1}_{a,b}(2,n-1)+bF^{1}_{a,b}(2,n-2), n \geq 1$} & \pbox{20cm}{$F^{1}_{a,b}(2,-1)=1$ \\ $F^{1}_{a,b}(2,0)=a$} & \pbox{20cm}{$p_{n}^{a,b}$, \\ \cite{benjamin2008tiling}} \\       
		\hline
		1 & 2 & $s$ & $t$ & \pbox{20cm}{$F^{1}_{s,t}(2,n) =$ \\ $sF^{1}_{s,t}(2,n-1)+tF^{1}_{s,t}(2,n-2), n \geq 1$} & \pbox{20cm}{$F^{1}_{s,t}(2,-1)=1$ \\ $F^{1}_{s,t}(2,0)=s$} & $\{n\}_{s,t}$, \cite{amdeberhan2014generalized} \\ \hline
		1 & $k$ & 1 & $t$ & \pbox{20cm}{$F^{1}_{1,t}(k,n) =$ \\ $F^{1}_{1,t}(k,n-1)+tF^{1}_{1,t}(k,n-k), n \geq k-1$} & \pbox{20cm}{$F^{1}_{1,t}(k,n)=1$, \\ $n=-1, \ldots, k\!-\!2$} & $J(k,t,n)$, \cite{szynal2014generalized} \\ \hline
		1 & 2 & $k$ & 1 & \pbox{20cm}{$F^{1}_{k,1}(2,n) =$ \\ $kF^{1}_{k,1}(2,n-1)+F^{1}_{k,1}(2,n-2), n \geq 1$} & \pbox{20cm}{$F^{1}_{k,1}(2,-1)=1$ \\ $F^{1}_{k,1}(2,0)=k$} & $F_{k,n}$, \cite{falcon2007fibonacci} \\ \hline
		1 & $k$ & 1 & 1 & \pbox{20cm}{$F^{1}_{1,1}(k,n) =$ \\ $F^{1}_{1,1}(k,n-1)+F^{1}_{1,1}(k,n-k), n \geq k-1$} & \pbox{20cm}{$F^{1}_{1,1}(k,n)=1$, \\ $n=-1, \ldots, k-2$} & \pbox{20cm}{$F_{1}(k,n)$, \cite{distance1} \\ $F_{1}(k,n)$ = \\ $F^{i}_{r,s}(k,n\!\!+\!k\!-\!2)$} \\ \hline
		2 & $k$ & 1 & 1 & \pbox{20cm}{$F^{2}_{1,1}(k,n) =$ \\ $F^{2}_{1,1}(k,n-2)+F^{2}_{1,1}(k,n-k), n \geq k-1$} & \pbox{20cm}{$F^{2}_{1,1}(k,n)=1$, \\ $n=-1, \ldots, k-2$} & $F_{2}(k,n)$, \cite{distance2} \\ \hline
		$k\!-\!1$ & $k$ & 1 & 1 & \pbox{20cm}{$F^{k-1}_{1,1}(k,n) =$ \\ $F^{k\!-\!1}_{1,1}(k,n\!-\!k\!+\!1)\!+\!F^{k-1}_{1,1}(k,n\!-\!k), n \geq k-1$} & \pbox{20cm}{$F^{k-1}_{1,1}(k,n)=1$, \\ $n=-1, \ldots, k-2$} & $F_{k-1}(k,n)$, \cite{distance} \\ \hline
		$r$ & 2 & $k$ & 1 & \pbox{20cm}{$F^{r}_{k,1}(2,n) =$ \\ $kF^{r}_{k,1}(2,n-r)+F^{r}_{k,1}(2,n-2), n \geq r-1$} & \pbox{20cm}{$F^{r}_{k,1}(2,n)=1$, \\ $n=-1, \ldots, r-2$} & $F_{k,n}(r)$, \cite{falcon2014generalized} \\ \hline
		1 & 2 & $k$ & 2 & \pbox{20cm}{$F^{1}_{k,1}(2,n) =$ \\ $kF^{1}_{k,1}(2,n-1)+2F^{1}_{k,1}(2,n-2), n \geq 1$} & \pbox{20cm}{$F^{1}_{k,1}(2,-1)=1$ \\ $F^{1}_{k,1}(2,0)=k$} & $J_{k,n}$, \cite{generatingjaco} \\ \hline
		1 & 2 & 2 & $k$ & \pbox{20cm}{$F^{1}_{2,k}(2,n) =$ \\ $2F^{1}_{2,k}(2,n-1)+kF^{1}_{2,k}(2,n-2), n \geq 1$} & \pbox{20cm}{$F^{1}_{2,k}(2,-1)=1$ \\ $F^{1}_{2,k}(2,0)=2$} & $P_{k}(n)$, \cite{catarino2013some} \\ \hline
		1 & c & 1 & 1 & \pbox{20cm}{$F^{1}_{1,1}(c,n) =$ \\ $F^{1}_{1,1}(c,n-1)+F^{1}_{1,1}(c,n-c), n \geq c-1$} & \pbox{20cm}{$F^{1}_{1,1}(c,n)=1$, \\ $n=-1, \ldots, c-2$} & $G_{n}$, \cite{bicknell1996classes} \\ \hline
	\end{tabular}
	\label{tab1}
	\caption{Some sequences generalized by $F^{i}_{r,s}(k,n)$ and $L^{i}_{r,s}(k,n)$}
\end{table}

\begin{remark}
	As we are considering that $i,k, r$, and $s$ are nonnegative integers, the polynomial generalization of the Fibonacci numbers $\{n\}_{s,t}$ presented in \cite{amdeberhan2014generalized} and the two-parameters sequence $p_{n}^{a,b}$ from \cite{benjamin2008tiling} coincide as we set $a=s$ and $b=t$. In Section \ref{Sec3} new identities for $\{n\}_{s,t}$, and so for $p_{n}^{a,b}$, are obtained.
\end{remark}

\begin{remark}
	In order to avoid confusion, the generalizations of the Fibonacci numbers presented in \cite{distance1}, \cite{distance2}, and \cite{distance} are denoted here by $F_{1}(k,n)$, $F_{s}(k,n)$, and $F_{k-1}(k,n)$, respectively.
\end{remark}

We finish this section by presenting a recurrence relation for $L_{r,s}^{i}(k,n)$.

\begin{theorem} Let $k > i$ and $n\geq 2k+(k-1)i-1$ be integers. Then, 
	$$L_{r,s}^{i}(k,n)=rL_{r,s}^{i}(k,n-i))+sL_{r,s}^{i}(k,n-k).$$
	\label{recor}
\end{theorem}

\begin{proof} Using \eqref{23.1} and the definition of $L_{r,s}^{i}(k,n)$, we have 
	$$\begin{array}{rcl}
	L_{r,s}^{i}(k,n) & = & (k-1)r^{k-1}sF_{r,s}^{i}(k,n-k-(k-1)i)+r^{k-1}F_{r,s}^{i}(k,n-(k-1)i) \\ & = & (k-1)r^{k}sF_{r,s}^{i}(k,n-k-(k-1)i-i) \\ & & +(k-1)r^{k-1}s^{2}F_{r,s}^{i}(k,n-k-(k-1)i-k) \\ & & +r^{k}F_{r,s}^{i}(k,n-(k-1)i-i)  +r^{k-1}sF_{r,s}^{i}(k,n-(k-1)i-k), \end{array}$$
	from where
	$$\begin{array}{rcl}
	L_{r,s}^{i}(k,n) & = & (k-1)r^{k}sF_{r,s}^{i}(k,n-k-ki)  +(k-1)r^{k-1}s^{2}F_{r,s}^{i}(k,n-2k-ki+i) \\ & & +r^{k}F_{r,s}^{i}(k,n-ki)  +r^{k-1}sF_{r,s}^{i}(k,n-ki+i-k). \end{array}$$
	and, then,
	$$\begin{array}{rcl}
	L_{r,s}^{i}(k,n) & = & (k-1)r^{k}sF_{r,s}^{i}(k,n-i-(k(i+1)-i)) \\ & & +r^{k}F_{r,s}^{i}(k,n-i-i(k-1)) \\ & & +(k-1)r^{k-1}s^{2}F_{r,s}^{i}(k,n-k-(k(i+1)-i)) \\ & & +r^{k-1}sF_{r,s}^{i}(k,n-k-(k-1)i) \\ & = & rL_{r,s}^{i}(k,n-i)+sL_{r,s}^{i}(k,n-k).
	\end{array}$$
\end{proof}

When $i=r=s=1$ and $k=2$, we get the recurrence relation defining the Lucas number: $L_n=L_{n-1}+L_{n-2}$.

\section{A combinatorial approach to $F^{i}_{r,s}(k,n)$ and $L^{i}_{r,s}(k,n)$}
\label{sec-citacao}

In this section, we give three combinatorial interpretations to the numbers $F^{i}_{r,s}(k,n)$ and $L^{i}_{r,s}(k,n)$. The well-known Benjamin and Quinn's \cite{benjamin2003} interpretation of the Fibonacci numbers in terms of tilings as well as the one studied in \cite{robson} are particular cases of one of our combinatorial interpretations. The combinatorial interpretation we provide for $L^{i}_{r,s}(k,n)$ in terms of tilings results, by taking $r=s=i=1$ and $k=2$, in a new combinatorial interpretation to the Lucas numbers. We also present a closed formula to $F^{i}_{r,s}(k,n)$ and the generating functions for $F^{i}_{r,s}(k,n)$ and $L^{i}_{r,s}(k,n)$.

\subsection{Counting tilings}
\label{sec2.1}

In order to present our combinatorial interpretation in terms of tilings, we consider $1 \times (n+1)$ boards, called $(n+1)$-boards, where, as usual, the first position is the leftmost square while the rightmost square occupies the $(n+1)$th position (see Figure \ref{fig1}).

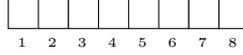
\begin{figure}[h!]
	\centering
	\begin{tikzpicture}[scale=0.8][rounded corners, ultra thick]
	\draw (-0.25,-0.75);
	\shade[top color=white,bottom color=white, draw=black] (0,-1.5) rectangle +(0.5,0.5) node[pos=-0.001,below right] {\tiny{$1$}};
	\shade[top color=white,bottom color=white, draw=black] (0.5,-1.5) rectangle +(0.5,0.5)node[pos=-0.001,below right] {\tiny{$2$}};
	\shade[top color=white,bottom color=white, draw=black] (1.0,-1.5) rectangle +(0.5,0.5)node[pos=-0.001,below right] {\tiny{$3$}};
	\shade[top color=white,bottom color=white, draw=black] (1.5,-1.5) rectangle +(0.5,0.5)node[pos=-0.001,below right] {\tiny{$4$}};
	\shade[top color=white,bottom color=white, draw=black] (2.0,-1.5) rectangle +(0.5,0.5)node[pos=-0.001,below right] {\tiny{$5$}};
	\shade[top color=white,bottom color=white, draw=black] (2.5,-1.5) rectangle +(0.5,0.5)node[pos=-0.001,below right] {\tiny{$6$}};
	\shade[top color=white,bottom color=white, draw=black] (3.0,-1.5) rectangle +(0.5,0.5)node[pos=-0.001,below right] {\tiny{$7$}};
	\shade[top color=white,bottom color=white, draw=black] (3.5,-1.5) rectangle +(0.5,0.5)node[pos=-0.001,below right] {\tiny{$8$}};
	\end{tikzpicture}
	\caption{An $8$-board}
	\label{fig1}	
\end{figure}


The pieces we are going to use in order to tile our boards are: $1\times 1$ black squares, $1\times i$ rectangles ($i$-rectangles), and $1\times k$ rectangles ($k$-rectangles). The $i$-rectangles and the $k$-rectangles may be colored with one of $r$ and $s$ colors, respectively.

Let $f^{i}_{r,s}(k,n)$ be the number of tilings of an $(n+1)$-board using the mentioned pieces such that black squares can only appear in the first $\min \lbrace i,k \rbrace -1$ positions. Let $n < \max \lbrace k,i \rbrace -1$. If $i \leq k$, then $-1\leq n \leq k-2$ and therefore, in this case, there do not exist tilings containing $k$-rectangles. As we can insert $\lfloor\frac{n+1}{i}\rfloor$ $i$-rectangles in this board and each  $i$-rectangle may be colored in $r$ different ways, we have $r^{\lfloor\frac{n+1}{i}\rfloor}$ tilings, i.e., $f^{i}_{r,s}(k,n)=r^{\lfloor\frac{n+1}{i}\rfloor}$. Analogously, if $k<i$ we have $f^{i}_{r,s}(k,n)= s^{\lfloor\frac{n+1}{k}\rfloor}$. If $n\geq \max \lbrace k,i \rbrace-1$, all tilings enumerated by $f^{i}_{r,s}(k,n)$ end either with an $i$-rectangle or a $k$-rectangle. Removing this last piece we are left either with $rf^{i}_{r,s}(k,n-i)$ or $sf^{i}_{r,s}(k,n-k)$ tilings, according to this last piece being an $i$-rectangle or a $k$-rectangle, respectively. Hence, $f^{i}_{r,s}(k,n)=rf^{i}_{r,s}(k,n-i)+sf^{i}_{r,s}(k,n-k)$.

As $f^{i}_{r,s}(k,n)$ and $F^{i}_{r,s}(k,n)$ have the same recurrence relation and share the same initial conditions, we have proved the following theorem which gives us a combinatorial interpretation for $F^{i}_{r,s}(k,n)$.

\begin{theorem}	
	The number of tilings of an $(n+1)$-board using $1 \times 1$ black squares, $i$-rectangles with one of $r$ colors and $k$-rectangles with one of $s$ colors, such that the black squares can only appear in the first $\min \lbrace i,k \rbrace -1$ positions, is equal to $F^{i}_{r,s}(k,n)$.
	\label{Interpret1}
\end{theorem}

\begin{example}
	Figure \ref{fig2} exhibits all tilings of an $8$-board under the conditions of Theorem \ref{Interpret1}, where $i=2$, $k=3$, and $r=s=1$. We emphasize that due to the restrictions we can have at most one black in the first position. In this case $F^{2}_{1,1}(3,7)=7$.
\end{example}

\begin{figure}[h!]
	\centering
	\begin{tikzpicture}[scale=0.8][rounded corners, ultra thick]
	
	\draw (-0.25,-0.75) ;
	\shade[top color=white,bottom color=white, draw=black] (0,-1) rectangle +(1.5,0.5);
	\shade[top color=white,bottom color=white, draw=black] (1.0,-1) rectangle +(1.5,0.5);
	\shade[top color=white,bottom color=white, draw=black] (2.0,-1) rectangle +(1.5,0.5);
	\shade[top color=white,bottom color=white, draw=black] (3.0,-1) rectangle +(1.0,0.5);
	\draw (-0.25,-0.75) ;
	\shade[top color=black,bottom color=black, draw=black] (0,-2) rectangle +(0.5,0.5);	
	\shade[top color=white,bottom color=white, draw=black] (0.5,-2) rectangle +(1.5,0.5);
	\shade[top color=white,bottom color=white, draw=black] (1.5,-2) rectangle +(1.0,0.5);
	\shade[top color=gray,bottom color=gray, draw=black] (2.5,-2) rectangle +(1.5,0.5);	
	\shade[top color=black,bottom color=black, draw=black] (0,-3) rectangle +(0.5,0.5);	
	\shade[top color=white,bottom color=white, draw=black] (0.5,-3) rectangle +(1.5,0.5);
	\shade[top color=gray,bottom color=gray, draw=black] (1.5,-3) rectangle +(1.5,0.5);
	\shade[top color=white,bottom color=white, draw=black] (3.0,-3) rectangle +(1.0,0.5);
	\shade[top color=black,bottom color=black, draw=black] (0,-4) rectangle +(0.5,0.5);	
	\shade[top color=gray,bottom color=gray, draw=black] (0.5,-4) rectangle +(1.5,0.5);
	\shade[top color=white,bottom color=white, draw=black] (2.0,-4) rectangle +(1.5,0.5);
	\shade[top color=white,bottom color=white, draw=black] (3.0,-4) rectangle +(1.0,0.5);		
	\shade[top color=white,bottom color=white, draw=black] (0,-5) rectangle +(1.5,0.5);
	\shade[top color=gray,bottom color=gray, draw=black] (1.0,-5) rectangle +(1.5,0.5);
	\shade[top color=gray,bottom color=gray, draw=black] (2.5,-5) rectangle +(1.5,0.5);
	\shade[top color=gray,bottom color=gray, draw=black] (0,-6) rectangle +(1.5,0.5);
	\shade[top color=white,bottom color=white, draw=black] (1.5,-6) rectangle +(1.5,0.5);
	\shade[top color=gray,bottom color=gray, draw=black] (2.5,-6) rectangle +(1.5,0.5);
	\shade[top color=gray,bottom color=gray, draw=black] (0,-7) rectangle +(1.5,0.5);
	\shade[top color=gray,bottom color=gray, draw=black] (1.5,-7) rectangle +(1.5,0.5);
	\shade[top color=white,bottom color=white, draw=black] (3.0,-7) rectangle +(1.0,0.5);
	\end{tikzpicture}
	\caption{The seven possible tilings enumerated by $F^{2}_{1,1}(3,7)$}
	\label{fig2}	
\end{figure}
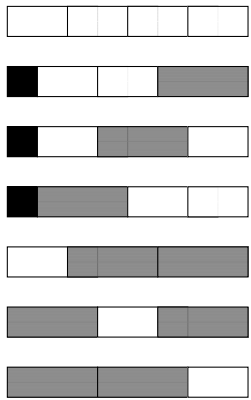 


\

The numbers $L^{i}_{r,s}(k,n)$ can now be interpreted as the number of those tilings described in Theorem \ref{Interpret1} having an additional conditions on the last pieces:
\begin{itemize}
	\item if $k-1\leq n\leq k+(k-1)i -2$, there exists at least one $k$-rectangle in the tiling,
	\item if $n\geq k+(k-1)i-1$, there are at least $k-1$ $i$-rectangles among the last $k$ pieces.
\end{itemize}
Indeed, when $n\geq k+(k-1)i-1$, the tilings ending with $k-1$ $i$-rectangles are enumerated by $r^{k-1}F^{i}_{r,s}(k,n-(k-1)i)$, while those whose last $k$ pieces contain a $k$-rectangle, between two or after all $k-1$ $i$-rectangles, are counted by $(k-1)r^{k-1}sF^{i}_{r,s}(k,n-k-(k-1)i)$, since there are $(k-1)$ possible positions to place the $k$-rectangle. If $k-1\leq n\leq k+(k-1)i-2$, it is easy to see that the number of tilings is given by $F^{i}_{r,s}(k,n)-r^{\lfloor \frac{n+1}{i}\rfloor}$, since exactly  $r^{\lfloor \frac{n+1}{i}\rfloor}$ tilings do not have any $k$-rectangle.



By taking $r=s=i=1$ and $k=2$ in $F^{i}_{r,s}(k,n)$ we obtains the well-known Benjamin and Quinn's \cite{benjamin2003} combinatorial interpretation of the Fibonacci numbers. In Chapter 2 of \cite{benjamin2003} the authors also presented an interpretation for the $n$th Lucas number, $L_n$, in terms of circular $n$-boards. The one we showed here for $L^{i}_{r,s}(k,n)$ results in a new combinatorial interpretation for $L_n$.

\subsection{A formula for $F^{i}_{r,s}(k,n)$}

As an application of the combinatorial interpretation discussed above, we present a closed formula for the number $F^{i}_{r,s}(k,n)$ involving its four parameters. As one can verify, the known formulas for $F_n$, $P_n$, $J_n$, $Pv(n)$, and $\{n \}_{s,t}$ are derived from our formula by specializing the four parameters. We also present a formula for the Narayana numbers $u_n$.



\begin{theorem} Let $k\geq i$ and $n\geq -1$ be integers. Then
	\begin{equation}
	F^{i}_{r,s}(k,n)=\sum\limits_{j=0}^{\lfloor \frac{n+1}{k} \rfloor}r^{\lfloor \frac{n+1-jk}{i}\rfloor}s^{j} \dbinom{j+\lfloor \frac{n+1-jk}{i} \rfloor}{j}.
	\label{r1}
	\end{equation} 
\end{theorem}

\begin{proof} We initially consider the case $n\geq k-1$. Let $j$ and $t$ be the number of $k$-rectangles and $i$-rectangles, respectively, in a tiling enumerated by $F^{i}_{r,s}(k,n)$. We have $0\leq j\leq \lfloor \frac{n+1}{k}\rfloor$, $0\leq t\leq \lfloor \frac{n+1}{i}\rfloor$, and, for each $j$, we also have $t=\lfloor\frac{n+1-jk}{i}\rfloor$. Considering the possible $s$ colors to the $k$-rectangles and $r$ colors to the $i$-rectangles, we conclude that the number of tilings having $j$ $k$-rectangles counted by $F^{i}_{r,s}(k,n)$ is equal to 
	$$r^{t}s^{j} \binom{j+t}{j}=r^{\lfloor\frac{n+1-jk}{i}\rfloor}s^{j}\dbinom{j+\lfloor \frac{n+1-jk}{i} \rfloor}{j}.$$ 
	Adding up relatively to the values of $j$, it follows that the number of tilings is given by the r.h.s. of \eqref{r1}. 
	
	When $n<k-1$, we have $\lfloor\frac{n+1}{k}\rfloor=0$ and, hence, 
	$$\sum\limits_{j=0}^{\lfloor \frac{n+1}{k} \rfloor}r^{\lfloor \frac{n+1-jk}{i}\rfloor}s^{j}\binom{j+\lfloor \frac{n+1-jk}{i} \rfloor}{j}=r^{\lfloor \frac{n+1}{i}\rfloor}.$$  
\end{proof} 

Besides the known formulas to the numbers $F_n$, $P_n$, $J_n$, $Pv(n)$, and $\{n \}_{s,t}$, we obtain an expression to calculate the $n$th Narayana number $u_n$ in the next corollary.



\begin{corollary}
	$$u_n= \sum \limits _{j=0}^{\lfloor \frac{n}{3}\rfloor} \dbinom{n-2j}{j}, \forall n\geq1.$$ 
\end{corollary}


\subsection{Two other combinatorial interpretations}

Let $k > i \geq 1$ and $n \geq i-1$ be integers. Let $X=\lbrace 1,2,\ldots,n+1 \rbrace$ and $\Im = \lbrace Y_t;t \in T\rbrace$ be a family of subsets of $X$ such that each $Y_t$, $t \in T$ contain consecutive integers and satisfies the following conditions:
\begin{enumerate}
	\item $|Y_{t}|\in \lbrace i,k \rbrace$,
	\item if $|Y_t|=i$, then the elements of $Y_t$ may be colored with one of $r$ different colors,
	\item if $|Y_t|=k$, then the elements of $Y_t$ may be colored with one of $s$ different colors,
	\item $|X \setminus \bigcup_{t\in T} Y_{t}|\leq i-1$,
	\item if $x \in X \setminus \bigcup_{t\in T} Y_{t}$, then $x\in \lbrace n+1,n,\ldots, n+3-i \rbrace$.
\end{enumerate}

The family $\Im$ is called \textit{color decomposition and remainder at most $i-1$ of $X$}. The set decomposition related to the Fibonacci sequence, see \cite{distance}, follows from this by taking $r=s=i=1$ and $k=2$. 

\begin{theorem}
	Let $k>i \geq 1$ and $n\geq i-1$ be integers. Then, the number of color decompositions with remainder at most $i-1$ of $X$ is equal to $F^{i}_{r,s}(k,n)$.
\end{theorem}

\begin{proof} Let $l(n)$ be the number of color decompositions with remainder at most $i-1$ of the set $X \lbrace 1,2,\ldots,n+1 \rbrace$. If $n<k-1$, then $l(n)=r^{\lfloor\frac{n+1}{i}\rfloor}$, since $|Y_t|=i, \forall t \in T$ and $Y_t$ may have one of the $r$ colors. Thus $l(n)=F^{i}_{r,s}(k,n)$ for $i-1\leq n < k-1$.
	
	Let us now suppose $n\geq k-1$ and $l(n)=F^{i}_{r,s}(k,n)$. We shaw show that  $l(n+1)=F^{i}_{r,s}(k,n+1)$. Let $l_i(n+1)$ be the number of all color decompositions with remainder at most $i-1$ of $X$ such that $\lbrace 1,2,\ldots,i \rbrace \in \Im$ and let $l_k(n+1)$ be the number of those decompositions $\Im$ such that $\lbrace1,2,\ldots, k \rbrace \in \Im$. If $\lbrace 1,2,\ldots,i \rbrace$ belongs to a decomposition of $X$, then $\lbrace1,2,\ldots, k \rbrace$ does not belong to this same decomposition and vice versa. It follows that $l(n+1)=l_i(n+1)+l_k(n+1)$. Besides that, we have $l_i(n+1)=rl(n+1-i)$ and $l_k(n+1)=sl(n+1-k)$. By induction and the recurrence relation \eqref{23.1} defining $F^{i}_{r,s}(k,n)$, we have
	$$\begin{array}{rcl}
	l(n+1) &=& rl(n+1-i)+sl(n+1-k) \\
	&=& rF^{i}_{r,s}(k,n+1-i)+sF^{i}_{r,s}(k,n+1-k)\\
	&=& F^{i}_{r,s}(k,n+1),
	\end{array}$$
	which concludes the proof.
\end{proof} 

\

We now present our third combinatorial interpretation for $F^{i}_{r,s}(k,n)$. Before doing so, we remember a few definitions. The total number of matchings in a graph (called Hosoya index or Z index) was introduced in 1971 \cite{hosoya1971topological} to study properties of organic compounds. Hosoya \cite{hosoya1973topological} also noticed the relation between Hosoya index and Fibonacci and Lucas sequences. Finally, in 1982, Prodinger and Tichy \cite{prodinger1982fibonacci} gave a complete graph interpretation to the Fibonacci sequence by exhibiting connections between the Fibonacci and Lucas numbers and the number of independent sets in some especial graphs. Some recent progress on interpreting  generalizations of the Fibonacci sequence in terms of graphs can be found in \cite{bednarz2015total,distance1,wloch2013some}. 

Let $X=\lbrace 1,2,\ldots, n+1 \rbrace$. If $X$ corresponds to the vertex set of the graph $P_{n+1}$, $n\geq i-1$, then each $Y_t$, $t \in T$ corresponds to a monochromatic subgraph $P_l$, where $l\in \lbrace i,k\rbrace$, i.e., there exist $r$ and $s$ colors available for $P_i$ and $P_k$, respectively. Thus, the color decomposition with rest ao most $i-1$ of $X$ corresponds to a $\lbrace P_k,P_i\rbrace$-matching of $P_n$. Obviously, at most the last $i-1$ vertex does not belong to the $\lbrace P_k,P_i\rbrace$-matching of $P_n$.

\subsection{The generating function for $F^{i}_{r,s}(k,n)$}

In this section we determine the generating function of the sequence ($F^{i}_{r,s}(k,n)$). We point out that our result generalizes those recently obtained to the sequences $k$-Jacobsthal \cite{generatingjaco}, $(k,r)$-Fibonacci numbers  \cite{falcon2014generalized}, the polynomial generalization of the Fibonacci numbers \cite{amdeberhan2014generalized}, and the $k$-Narayana numbers \cite{knarayana}.

\begin{proposition} The generating function of $F^{i}_{r,s}(k,n)$ is given by
	\begin{equation*}
	f^{i}_{r,s}(k,n,x)=
	\left \{
	\begin{array}{ll}
	\frac{1+r^{\lfloor\frac{1}{i}\rfloor}x+\cdots+r^{\lfloor\frac{k-1}{i}\rfloor}x^{k-1}-r(x^i+r^{\lfloor\frac{1}{i}\rfloor}x^{i+1}+\cdots+r^{\lfloor\frac{k-1-i}{i}\rfloor}x^{k-1})}{1-rx^i-sx^k}, & \textnormal{if} \ k>i \\
	\\
	\frac{1+r^{\lfloor\frac{1}{i}\rfloor}x+\cdots+r^{\lfloor\frac{k-1}{i}\rfloor}x^{k-1}}{1-x^k(r+s)}, & \textnormal{if} \ k=i. \\
	\end{array}
	\right.
	\end{equation*}
\end{proposition}

\begin{proof} Let $k>i$. Define $g^{i}_{r,s}(k,n,x)$ as the series $g^{i}_{r,s}(k,n,x)=F^{i}_{r,s}(k,-1)+F^{i}_{r,s}(k,0)x+\cdots+F^{i}_{r,s}(k,n-1)x^n + \cdots$. Then,
	$$\begin{array}{rcl}
	rx^ig^{i}_{r,s}(k,n,x) &=& rF^{i}_{r,s}(k,-1)x^i+\cdots+rF^{i}_{r,s}(k,n-1)x^{n+i} + \cdots  \\
	sx^kg^{i}_{r,s}(k,n,x) &=& sF^{i}_{r,s}(k,-1)x^k+\cdots+sF^{i}_{r,s}(k,n-1)x^{n+k} + \cdots 
	\end{array}$$
	Since $F^{i}_{r,s}(k,n)=rF^{i}_{r,s}(k,n-i)+sF^{i}_{r,s}(k,n-k)$, if $n\geq k-1$, and $F^{i}_{r,s}(k,n)=r^{\lfloor\frac{n+1}{i}\rfloor}$, if $n<k-1$, it follows that $(1-rx^i-sx^k)g^{i}_{r,s}(k,n,x)=1+r^{\lfloor\frac{1}{i}\rfloor}x+\cdots+r^{\lfloor\frac{k-1}{i}\rfloor}x^{k-1}-r(x^i+r^{\lfloor\frac{1}{i}\rfloor}x^{i+1}+\cdots+r^{\lfloor\frac{k-1-i}{i}\rfloor}x^{k-1})$. Thus  $f^{i}_{r,s}(k,n,x)=g^{i}_{r,s}(k,n,x)$. The case $k=i$ is analogous.
\end{proof}

\

\section{Identities involving $F^{i}_{r,s}(k,n)$ and $L^{i}_{r,s}(k,n)$}
\label{Sec3}

This section is devoted to presenting a great number of identities involving $F^{i}_{r,s}(k,n)$ and $L^{i}_{r,s}(k,n)$. Although all these identities could be proved by induction or algebraically, we employ the combinatorial interpretation presented in Section \ref{sec-citacao}. In what follows, we call type $\mathcal{F}$ and type $\mathcal{L}$ the tilings enumerated by $F^{i}_{r,s}(k,n)$ and $L^{i}_{r,s}(k,n)$, respectively.

When we specialize the four parameters $i,k,r$, and $s$ several identities involving those numbers generalized by $F^{i}_{r,s}(k,n)$ and $L^{i}_{r,s}(k,n)$ appear. The new ones will be stated as corollaries while the known ones will just be mentioned after the theorems.


\begin{theorem} Let $i\geq 2$, $k=ti$, and $n\geq 0$ be integers, where $t\geq 1$. Then
	$$F^{i}_{r,s}(k,ni-1)=F^{i}_{r,s}(k,ni).$$
\end{theorem}

\begin{proof} We have $F^{i}_{r,s}(k,ni-1)$ type $\mathcal{F}$ tilings of an $ni$-board. Since $k=ti$, there do not exist tilings with a black square. By inserting a black square in the beginning of such tilings, we are left with all those tilings enumerated by $F^{i}_{r,s}(k,ni)$. Reciprocally, each tiling enumerated by $F^{i}_{r,s}(k,ni)$ has exactly one black square, which, after being removed, is turned into a tiling counted by $F^{i}_{r,s}(k,ni-1)$.
\end{proof}

\begin{theorem} Let $n\geq 0$ and $i,k\geq 1$ be integers. Then, 
	\begin{equation}
	F^{i}_{r,s}(k,ni+k)=
	\left \{
	\begin{array}{cc}
	r^{ n+k+1}+\sum\limits_{j=-1}^{n} r^{n-j}sF^{i}_{r,s}(k,j), & \textnormal{if} \ i=1 \\
	u(k,i)+\sum\limits_{j=0}^{n} r^{n-j}sF^{i}_{r,s}(k,ij), & \textnormal{if} \ i\neq 1, \\
	\end{array}
	\right.
	\label{r2}
	\end{equation} where
	\begin{equation*}
	u(k,i)=
	\left \{
	\begin{array}{lll}
	r^{\lfloor \frac{ni+k+1}{i}\rfloor}, & \textnormal{if} \ k\geq i \\
	r^{n+1}, & \textnormal{if} \  i=k+1 \\
	0, & \textnormal{if} \  i\geq k+2. \\
	\end{array}
	\right.
	\end{equation*}
	\label{9.1}
\end{theorem}

\begin{proof}	We shall show that the r.h.s. of \eqref{r2} also counts the number of type $\mathcal{F}$ tilings of an $(ni+k+1)$-board. In order to achieve this, we consider the right most appearance of a $k$-rectangle, if any. 
	
	The number of tilings without any $k$-rectangle is equal to either $r^{\lfloor \frac{ni+k+1}{i}\rfloor}$, if $k\geq i$, or $r^{n+1}$, if $i=k+1$, or $0$, if $i\geq k+2$. The remaining tilings end with either a $k$-rectangle or a $k$-rectangle followed by $i$-rectangles. There exist $sF^{i}_{r,s}(k,ni)$ type $\mathcal{F}$ tilings whose last piece is a $k$-rectangle, $rsF^{i}_{r,s}(k,(n-1)i)$ ending with a $k$-rectangle followed by one $i$-rectangle, $r^{2}sF^{i}_{r,s}(k,(n-2)i)$ ending with a $k$-rectangle followed by two $i$-rectangles, $\ldots$ , $r^{n-1}sF^{i}_{r,s}(k,i)$ ending with a $k$-rectangle followed by $(n-1)$ $i$-rectangles and, finally, $r^{n}sF^{i}_{r,s}(k,0)$ ending with a $k$-rectangle followed by $n$ $i$-rectangles. In the case $i=1$, we have one tiling ending with a $k$-rectangle followed by $n+1$  $i$-rectangles. By adding up all these possibilities, we finish the proof. 
\end{proof}

It follows from this theorem some well-known identities (see Table \ref{tab1} for the choice of the parameters) as well as new ones:
$$\centering \begin{array}{lll}
\sum \limits_{j=1}^{n}F_j=F_{n+2}-1 & \sum \limits_{j=0}^{n}F_{2j}=F_{2n+1}-1 & \sum \limits_{j=0}^{n}P_{2j}=\frac{1}{2}(P_{2n+1}\!-\!1) \\
\sum \limits_{j=2}^{n}J_j=\frac{1}{2}(J_{n+2}\!-\!3) & \sum \limits_{j=0}^{n}Pv(2j\!+\!1)=Pv(2n\!+\!4)\!-\!1 & \sum \limits_{j=1}^{n}u_j=u_{n+3}-1, \mbox{\cite{flaut2013generalized}}.
\end{array}$$


\begin{corollary}  \textit{}
	\begin{enumerate}
		\setlength\itemsep{-0.1cm}
		\item $\sum \limits_{j=1}^{n}2^{-j}J_{2j}=\frac{1}{2^n}(J_{2n+1}-2^n), \forall n \geq 1$,
		\item $\sum \limits_{j=2}^{n}2^{-j}P_{j}=\frac{1}{2^n}(P_{n+2}-2^{n+1}-2^{n-1}), \forall n \geq 2$.
		\item $\{n+2\}_{s,t} = s^{n+1} + \sum \limits_{j=-1}^{n-2}s^{n-2-j}t\{2+j\}_{s,t}, \forall n \geq 0$.  
	\end{enumerate}
\end{corollary}

The proof of the next identity is similar to that of the previous theorem and will be omitted.

\begin{theorem} Let $n\geq 0$ and $k,i\geq 1$ be integers. Then,
	$$F^{i}_{r,s}(k,ni+k+1) = u(k,i)+ \sum\limits_{j=l}^{n}r^{n-j}sF^{i}_{r,s}(k,ji+1),$$
	where
	\begin{equation*}
	l=
	\left \{
	\begin{array}{cc}
	-2, & \mbox{if} \ i=1 \\
	-1, & \mbox{if} \  i=2 \\
	0  , & \mbox{if} \  i\geq3 \\
	\end{array}
	\right.
	\\
	\mbox{\ and \ }
	v(k,i)=\left \{
	\begin{array}{lll}
	r^{\lfloor \frac{ni+k+2}{i}\rfloor}, & \mbox{if} \ k\geq i  \\
	r^{n+1},    & \mbox{if} \ \ i=k+1 \ \mbox{and} \ i\geq3 \  \mbox{or} \ i=k+2 \\
	0,    & \mbox{if} \ i\geq k+3 \  \mbox{or} \ i=2 \ \mbox{and} \ k=1.  
	\end{array}
	\right.
	\end{equation*}
\end{theorem}
From this identity we obtain a known relation to the Padovan numbers,
$\sum \limits_{j=0}^{n}Pv(2j)=Pv(2n+3)-1,$
and also a new identity to the polynomial generalization of the Fibonacci numbers $\{ n \}_{s,t}$:

\begin{corollary} 
	$$\{n+5\}_{s,t} = s^{n+4} + \sum \limits_{j=-2}^{n}s^{n-j}t\{3+j\}_{s,t}, \forall n \geq 0.$$  
\end{corollary}

\begin{theorem} Let $n\geq 0$, $k\geq i+1\geq2$, and $-1\leq m\leq k-(i+1)$ be integers. Then,
	\begin{equation}
	F^{i}_{r,s}(k,nk+m+i)=
	\left \{
	\begin{array}{ll}
	s^{n+1}+\sum\limits_{j=0}^{n}rs^{n-j}F^{i}_{r,s}(k,jk+m), & \textnormal{if} \ m+i+1=k \\
	\sum\limits_{j=0}^{n}rs^{n-j}F^{i}_{r,s}(k,jk+m), & \textnormal{if} \ m+i+1< k. \\
	\end{array}
	\right.
	\label{r3}
	\end{equation}
\end{theorem}

\begin{proof} We prove that the r.h.s. of \eqref{r3} counts the same tilings as $F^{i}_{r,s}(k,nk+m+i)$ by considering the rightmost appearance of an $i$-rectangle, if any. When $m+i+1=k$ there are $s^{n+1}$ tilings without $i$-rectangle. In the other cases all type $\mathcal{F}$ tilings contains at least one $i$-rectangle, since $m+i+1<k$. Now, following the same steps of the proof of Theorem \ref{9.1}, we obtain \eqref{r3}.
	
\end{proof}

From \eqref{r3} we can derive some known identities and new ones:
$$\centering 
\begin{array}{lll}
\sum \limits_{j=0}^{n}F_{2j+1}=F_{2n+2} & u_{3n}= \sum \limits_{j=0}^{n-1}u_{3j+2} & Pv(3n+2)=\sum \limits_{j=0}^{n}Pv(3j) \\
\sum \limits_{j=0}^{n}2^{n-j}J_{2j+1}=J_{2n+2} & u_{3n-1}= \sum \limits_{j=0}^{n-1}u_{3j+1} & Pv(3n+3)\!-\!1=\sum \limits_{j=0}^{n}Pv(3j\!+\!1)  \\
2\sum \limits_{j=0}^{n}P_{2j+1}=P_{2n+2} & u_{3n+1}-1= \sum \limits_{j=1}^{n}u_{3j} & 
\end{array}$$


\begin{corollary}  \textit{}
	\begin{enumerate}
		\setlength\itemsep{-0.1cm}
		\item $\{2n+2\}_{s,t} = \sum \limits_{j=0}^{n}st^{n-j}\{2j+1\}_{s,t}$, $\forall n \geq 0$.  
		\item $\{2n+3\}_{s,t} = t^{n+1} + \sum \limits_{j=0}^{n}st^{n-j}\{2j+2\}_{s,t}$, $\forall n \geq 0$.
	\end{enumerate}
\end{corollary}


The same strategy of the proof of the last theorem can be employed to prove the next theorem, where again we have to consider the rightmost appearance of an $i$-rectangle, if any. 

\begin{theorem}	Let $n\geq 0$, $k>i\geq 1$, and $k-i\leq m < k$ be integers. Then,
	$$F^{i}_{r,s}(k,nk+m+i)=u(m)+\sum\limits_{j=0}^{n}rs^{n-j}F^{i}_{r,s}(k,jk+m),$$
	where
	\begin{equation}
	u(m)=
	\left \{
	\begin{array}{ll}
	s^{n+1}, & \textnormal{if} \ k-i\leq m\leq k-2 \\
	rs^{n+1}, & \textnormal{if} \ m=k-1. \\
	\end{array}
	\right.
	\label{r4}
	\end{equation}
\end{theorem}


From \eqref{r4} we get two known identities involving the Padovan numbers: \linebreak $Pv(3n+4)-1=\sum \limits_{j=0}^{n}Pv(3j+2)$ and $Pv(3n+5)-1=\sum \limits_{j=0}^{n}Pv(3j+3)$
and a new one for $\{n\}_{s,t}$:
\begin{corollary}
	$$\{2n+4\}_{s,t} = st^{n+1} + \sum \limits_{j=0}^{n}st^{n-j}\{2j+3\}_{s,t}, \forall n \geq 0.$$
\end{corollary}


\begin{theorem} Let $k,i\geq1$ and $n\geq \max \lbrace i-1, 2k-i-1 \rbrace$ be integers. Then, 
	$$\begin{array}{rl} F^{i}_{r,s}(k,n)  = & (s-r)F^{i}_{r,s}(k,n-k)+rF^{i}_{r,s}(k,n-i) \\ & +F^{i}_{r,s}(k,n-k+i)-sF^{i}_{r,s}(k,n-2k+i).\end{array}$$
\end{theorem}

\begin{proof} This result follows directly from \eqref{23.1}, the recurrence relation defining $F^{i}_{r,s}(k,n)$.
\end{proof}

As a consequence, we have the known identity $F_{n+2}=\frac{1}{2}(F_n+F_{n+3})$
and two new ones:
\begin{corollary}  \textit{}
	\begin{enumerate}
		\setlength\itemsep{-0.1cm}
		\item $\{n+2\}_{s,t} = (t-s)\{n\}_{s,t} + (s+1)\{n+1\}_{s,t} - t\{n-1\}_{s,t}, \forall n \geq 2$.  
		\item $u_{n+2}=\frac{1}{2}(u_{n-1}+u_{n+4}-u_{n}), \forall n \geq 2$.
	\end{enumerate}
\end{corollary}


\begin{theorem}	Let $t\geq1$ and $n\geq 2t-1$ be integers. Then,
	\begin{equation}
	F^{1}_{r,s}(2t,2n)=\sum\limits_{j\geq 0}\sum\limits_{l\geq0}r^{2n+1-2t(j+l)}s^{j+l} \dbinom{n-t(j+l)+j}{j}\dbinom{n-t(j+l)+l}{l}.
	\label{eq23.1}
	\end{equation}
\end{theorem}

\begin{proof} We show that the r.h.s. of \eqref{23.1} is equal to $F^{1}_{r,s}(2t,2n)$ by considering the number of $2t$-rectangles in each side of the \textit{central square}, i.e. the square of the $(2n+1)$-board that has the same number of squares on both left and right sides. Let $j$ and $l$ be the number of $2t$-rectangles to the left and to the right of the central square, respectively. There are  $j+l$ \ $2t$-rectangles and $2n+1-2t(j+l)$ $1$-rectangles in the tiling and, hence, we have $n-t(j+l)$ squares in each side of the central square. As the left side has $n-t(j+l)+j$ pieces, where $j$ are $2t$-rectangles, then we have $r^{n-t(j+l)}s^{j}\dbinom{n-t(j+l)+j}{j}$ different ways to color this left side. A similar reasoning shows that there are $r^{n-t(j+l)}s^{l}\dbinom{n-t(j+l)+l}{l}$ choices to the right side of the central square. Varying $j$ and $l$, we will have counted all possible tilings.
\end{proof}

Taking $t=r=s=1$ it follows the known identity 
$$F_{2n}=\sum\limits_{j\geq 0}\sum\limits_{l\geq 0}\dbinom{n-l-1}{j}\dbinom{n-j-1}{l}, \forall \ n\geq 1,$$ see \cite{benjamin2003}. We can also derive similar results to Pell and Jacobsthal number as well as to the sequence $\{n\}_{s,t}$:

\begin{corollary}  \textit{}
	\begin{enumerate}
		\setlength\itemsep{-0.05cm}
		\item {$P_{2n}=\sum\limits_{j\geq 0}\sum\limits_{l\geq 0}2^{2n-1-2(j+l)}\dbinom{n-l-1}{j}\dbinom{n-j-1}{l}, \forall \ n\geq 1$ } 
		\item {$J_{2n}=\sum\limits_{j\geq 0}\sum\limits_{l\geq 0}2^{j+l}\dbinom{n-l-1}{j}\dbinom{n-j-1}{l}, \forall \ n\geq 1$}
		\item {$\{2n\}_{s,t}=\sum\limits_{j\geq 0}\sum\limits_{l\geq 0}s^{2n-1-2(j+l)}t^{j+l}\dbinom{n-l-1}{j}\dbinom{n-j-1}{l}, \forall \ n\geq 1$.}
	\end{enumerate}
\end{corollary}


The next identity provides a simple proof of the identity 
\begin{equation}
F_n=F_{n-4}+4F_{n-5}+6F_{n-6}+4F_{n-7}+F_{n-8}, \forall n\geq 9
\label{30.1}
\end{equation} 
and gives us new relations involving Fibonacci, Pell, Jacobsthal and Padovan numbers.

\begin{theorem} Let $l\geq0$ and $n\geq  \max \lbrace i,k\rbrace l-1$ be integers. Then,
	$$F^{i}_{r,s}(k,n)=\sum\limits_{j\geq 0}r^{l-j}s^{j} \dbinom{l}{j}F^{i}_{r,s}(k,n-jk-(l-j)i).$$
\end{theorem}

\begin{proof} We count the number of type $\mathcal{F}$ tilings of an $(n+1)$-board by considering the last $l$ pieces. Initially we choose $j$ $k$-rectangles among the last $l$ pieces, where $0\leq j \leq l$, which can be done in $\binom{l}{j}$ ways. By removing the last $j$ $k$-rectangles together with the $l-j$ $i$-rectangles, we are left with tilings enumerated by $r^{l-j}s^jF^{i}_{r,s}(k,n-jk-(l-j)i)$. By adding up on all $j \geq 0$, we finish the proof.
\end{proof}

By taking $l=1$ in this identity we get the recurrence relation in \eqref{23.1}.

\begin{corollary} Let $l\geq0$ and $n\geq 2l+1$ be integers. Then,  
	\begin{multicols}{2}
		\begin{enumerate}
			\setlength\itemsep{0.0001cm}
			\item $F_n= \sum \limits_{j\geq 0} \dbinom{l}{j}F_{n-j-l}$
			\item $P_n= \sum \limits_{j\geq 0} \dbinom{l}{j}2^{l-j}P_{n-j-l}$
			\item $J_n= \sum \limits_{j\geq 0} \dbinom{l}{j}2^{j}J_{n-j-l}$
			\item $\{n\}_{s,t} = \sum \limits_{j \geq 0}s^{l-j}t^{j}\{n\!-\!j\!-\!l\}_{s,t}$.
			\item $Pv(n)\! \!=\!\! \sum \limits_{j\geq 0} \dbinom{l}{j}Pv(n\!-\!j\!-\!2l), \forall n \geq 3.$
			\item $u_n\!=\! \sum \limits_{j\geq 0} \dbinom{l}{j}u_{n\!-\!2j\!-\!l}, \forall n\geq 3l+1.$
		\end{enumerate}
	\end{multicols}
\end{corollary}

If we take $l=4$ in the Identity $1$ of the above corollary, we obtain \eqref{30.1}. We also have, by taking $l=m$ and $n=2m+1$, the well-known identity: 
$F_{2m+1}= \sum \limits_{j\geq 0} \dbinom{m}{j}F_{j+1}.$  

\begin{theorem}	Let $i,k\geq 1$ and $n\geq k-1$ be integers, such that $k=2i$. Then,
	\begin{equation} 
	(2+r^{2})F_{r,1}^{i}(k,n)=F_{r,1}^{i}(k,n+k)+F_{r,1}^{i}(k,n-k).
	\label{31.1}
	\end{equation}
\end{theorem}

\begin{proof} In order to prove this identity we establish a $1$-to-$(2+r^{2})$ correspondence between type $\mathcal{F}$ tilings of an $(n+1)$-board and either type $\mathcal{F}$ tilings of an $(n+k+1)$-board or type $\mathcal{F}$ tilings of an $(n-k+1)$-board, i.e., from each tiling enumerated by $F_{r,1}^{i}(k,n)$ we create $2+r^{2}$ tilings counted by either by $F_{r,1}^{i}(k,n+k)$ or by $F_{r,1}^{i}(k,n-k)$, in a unique way.
	
	Given a tiling enumerated by $F_{r,1}^{i}(k,n+k)$, we can determine a type $\mathcal{F}$ tilings of an $(n+1)$-board examining the last pieces and removing:
	\begin{itemize}
		\setlength\itemsep{-0.2cm}
		\item[(i)] the last $k$-rectangle, if it ends with one $k$-rectangle,
		\item[(ii)] the last two $i$-rectangles, if it ends with two or more $i$-rectangles or 
		\item[(iii)] the last $k$-rectangle, if it ends with one $i$-rectangle preceded by one $k$-rectangle.
	\end{itemize}
	
	In Case (i) we obtain a unique tiling enumerated by $F_{r,1}^{i}(k,n)$. Case (ii) results in $r^{2}F_{r,1}^{i}(k,n)$ tilings while Case (iii) gives us tilings counted by $F_{r,1}^{i}(k,n)$ ending with one $i$-rectangle. To finish the proof of \eqref{31.1}, we notice that the remaining tilings enumerated by $F_{r,1}^{i}(k,n)$ end with one $k$-rectangle and come from those tilings counted by $F_{r,1}^{i}(k,n-k)$ by just adding a $k$-rectangle at the end. 
	
	Reciprocally, for each tiling enumerated by $F_{r,1}^{i}(k,n)$ we can produce a type $\mathcal{F}$ tiling of an $(n+k+1)$-board either by adding a $k$-rectangle or adding two $i$-rectangles at the end. In the last case, we get $r^{2}$ type $\mathcal{F}$ tilings of an $(n+k+1)$-board, since all $i$-rectangle can be colored using one of $r$ colors and, in the former case, one tiling, since $s=1$. If a type $\mathcal{F}$ tiling of an $(n+1)$-board ends with an $i$-rectangle, we insert a $k$-rectangle before the last piece in order to obtain a type $\mathcal{F}$ tiling of an $(n+k+1)$-board. Finally, if a type $\mathcal{F}$ tilingof an $(n+1)$-board ends with a $k$-rectangle, then we remove this piece to create a tiling enumerated by $F_{r,1}^{i}(k,n-k)$.
\end{proof}

Two known identities to the Fibonacci and Pell numbers can be obtained from \eqref{31.1}: $3F_{n}\!=\!F_{n+2}\!+\!F_{n-2}$ (see \cite{benjamin2008tiling}) and $6P_{n}\!=\!P_{n+2}\!+\!P_{n-2}$,    (see \cite{koshpell}), $\forall n\! \geq \!2$.

\begin{theorem}	Let $i,k \geq 1$ and $n\geq max\lbrace i,k \rbrace-1$ be integers. Let $l$ be the remainder of the division of $n+1$ by $k$. Then,
	\begin{equation}
	F_{r,s}^{i}(k,n)=t(i,k,n)+\sum\limits_{j=0}^{\lfloor \frac{n+1-i}{k} \rfloor} 
	rs^{j}F_{r,s}^{i}(k,n-i-jk),
	\label{31.2}
	\end{equation}
	where
	\begin{equation*}
	t(i,k,n)=
	\left \{ 
	\begin{array}{ll}
	s^{\lfloor\frac{n+1}{k}\rfloor}, & \textnormal{se} \ i\geq k \ \textnormal{ou} \ \ k>i \ \textnormal{e} \ l \leq i-1 \\
	0, & \textnormal{se} \ k>i \ \textnormal{e} \ l > i-1. \\
	\end{array}
	\right.
	\end{equation*}
\end{theorem}

\begin{proof} We count the number of type $\mathcal{F}$ tilings of an $n$-board by considering the rightmost appearance of an $i$-rectangle, if any. We note that $t(i,k,n)$ is the number of tilings without $i$-rectangles. Indeed, if $i\geq k$ or $k>i$ and $l \leq i-1$ we immediately have  $t(i,k,n)=s^{\lfloor\frac{n+1}{k}\rfloor}$, while the case $k>i$ and $l > i-1$ does not gives us any tiling without $i$-rectangles. The remaining tilings end with an $i$-rectangle followed by $k$-rectangles. There are $rF_{r,s}^{i}(k,n-i)$ tilings whose last piece is an $i$-rectangle, $rsF_{r,s}^{i}(k,n-i-k)$ ending with an $i$-rectangle followed by one $k$-rectangle, $rs^{2}F_{r,s}^{i}(k,n-i-2k)$ ending with an $i$-rectangle followed by two $k$-rectangles, $\ldots$ , $rs^{\lfloor \frac{n+1-i}{k} \rfloor}F_{r,s}^{i}(k,n-i-\lfloor \frac{n+1-i}{k} \rfloor k )$ ending with an $i$-rectangle followed by $\lfloor \frac{n+1-i}{k} \rfloor$ \ $k$-rectangles. By adding up all these possibilities, we complete the proof. 
\end{proof}

A new identity to the Jacobsthal numbers is obtained from \eqref{31.2}:

\begin{corollary}
	$$J_{n}=1+2\sum\limits_{j=1}^{n-2}J_{j}, \forall n\geq 3.$$
\end{corollary}


\begin{theorem} Let $i,k\geq 1$ and $n\geq max\lbrace k-i-1,-1 \rbrace$ be integers. Then, 
	$$F_{r,s}^{i}(k,n+ik)=r^{k}F_{r,s}^{i}(k,n)+\sum\limits_{j=0}^{k-1}r^{j}sF_{r,s}^{i}(k,n+ik-k-ji).$$ 
\end{theorem}

\begin{proof} The number of type $\mathcal{F}$ tilings enumerated by $F_{r,s}^{i}(k,n+ik)$ ending with $k$ or more $i$-rectangles equals $r^{k}F_{r,s}^{i}(k,n)$. The remaining tilings end with exactly  $\textit{j}$ $i$-rectangles, where $0 \leq j \leq k-1$. By removing these $\textit{j}$ $i$-rectangles and the last $k$-rectangle we are left with tilings enumerated by $r^{j}sF_{r,s}^{i}(k,n+ik-k-ji)$. In order to finish the proof, we just have to add up on $j$.
\end{proof}


\begin{remark} From this theorem we have the following identity for the Padovan numbers:
	$$Pv(n+7)=Pv(n+1)+Pv(n)+Pv(n+2)+Pv(n+4),$$
	from where, by using the recurrence relation definig $Pv(n)$, it is easy to see that the period of the Padovan numbers modulus 2 is 7.
\end{remark}


\begin{theorem} Let $i,k,m\geq 1$ and $n\geq max\lbrace i-k-1, -1 \rbrace$ be integers. Then,
	$$F_{r,s}^{i}(k,n+mk)=s^{m}F_{r,s}^{i}(k,n)+\sum\limits_{j=0}^{m-1}rs^{j}F_{r,s}^{i}(k,n+mk-jk-i).$$
\end{theorem}

\begin{proof} The number of type $\mathcal{F}$ tilings enumerated by $F_{r,s}^{i}(k,n+mk)$ ending with at least $m$ $k$-rectangles equals $s^{m}F_{r,s}^{i}(k,n)$. The proof follows as in the previous theorem.
\end{proof}


\begin{corollary}
	\textit{}
	\begin{enumerate}
		\setlength\itemsep{-0.05cm}
		\item $F_{n+m}=F_{n+2}+\sum\limits_{j=2}^{m-1}F_{n+m-j}$,  $\forall n\geq 0, m\geq 3$,
		\item $P_{n+m}=2^{m-2}P_{n+2}+\sum\limits_{j=2}^{m-1}2^{j-2}P_{n+m-j}$, $\forall n\geq 0, m\geq 3$,
		\item $Pv(n+3m)=Pv(n)+\sum\limits_{j=0}^{m-1}Pv(n+3m-3j-2)$, $\forall n, m\geq 1$,
		\item $J_{n+m}=J_{n+2}+2\sum\limits_{j=2}^{m-1}J_{n+m-j}$, $\forall n\geq 0, m\geq 3$,
		\item $u_{n+3m}=u_{n}+\sum\limits_{j=0}^{m-1}u_{n-1+3m-3j}$, $\forall n\geq 0, m\geq 1$,
		\item $\{n+2m\}_{s,t}=t^{m}\{n\}_{s,t}+\sum\limits_{j=0}^{m-1}st^{j}\{n+2m-2j-1\}_{s,t}$, $\forall n\geq 0, m\geq 1$.
	\end{enumerate}
\end{corollary}


The next theorem has, as a consequence, a great number of identities involving sequences generalized by $F_{r,s}^{i}(k,n)$.

\begin{theorem}	Let $i,k\geq 1$ and $n\geq 2max\lbrace i,k \rbrace-1$ be integers. Let $l$ be the remainder of the division of $n+1$ by $i$. Then,
	\begin{equation}
	F_{r,s}^{i}(k,n)=u(i,k,n)+\sum\limits_{j=0}^{\lfloor \frac{n+1-2k}{i}\rfloor}\sum\limits_{t=0}^{\lfloor \frac{n+1-2k-ij}{i}\rfloor}s^{2}r^{j+t}F_{\lbrace r,s \rbrace}^{(i)}(k,n-2k-ti-ji),
	\label{31.3}
	\end{equation}
	where
	\begin{equation*}
	u(i,k,n)=
	\left \{ 
	\begin{array}{ll}
	r^{\lfloor\frac{n+1}{i}\rfloor}\!+\!sr^{\lfloor \frac{n+1-k}{i}\rfloor}(\lfloor \frac{n+1-k}{i}\rfloor\!+\!1), & \textnormal{\textit{if}} \ k\geq i \ \textnormal{\textit{or}} \ i>k, \  l\leq 2k\!-\!i\!-\!1\\
	& \\
	r^{\lfloor\frac{n+1}{i}\rfloor} & \textnormal{\textit{if}} \ i>k \ \textnormal{\textit{and}} \ 2k\!-\!i\!-\!1 < l \leq k\!-\!1\\			
	&\\
	sr^{\lfloor \frac{n+1-k}{i}\rfloor}(\lfloor \frac{n+1-k}{i}\rfloor+1), & \textnormal{\textit{if}} \ i>k \ \textnormal{\textit{and}} \ k\!-\!1<l\leq 2k\!-\!1\\
	&\\
	0,& \textnormal{\textit{if}} \ i>k \  \textnormal{\textit{and}} \ l>2k\!-\!1.\\
	\end{array}
	\right.
	\end{equation*}
\end{theorem}

\begin{proof} We shaw show that the r.h.s. of \eqref{31.3} enumerates the same tilings as $F_{r,s}^{i}(k,n)$, considering the number of $i$-rectangles between the last two $k$-rectangles.
	
	When $k\geq i$, it is clear that there are $r^{\lfloor\frac{n+1}{i}\rfloor}$ type $\mathcal{F}$ tilings of an $(n+1)$-board not having any $k$-rectangle. The number of tilings having exactly one $k$-rectangle is $(\lfloor \frac{n+1-k}{i}\rfloor+1)sr^{\lfloor \frac{n+1-k}{i}\rfloor}$. Indeed, a tiling having one  $k$-rectangle has $\lfloor \frac{n+1-k}{i}\rfloor$ $i$-rectangles, thus the $k$-rectangle may be inserted in one of $\lfloor \frac{n+1-k}{i} \rfloor +1$ positions. Now, we consider the case $i>k$ and $l\leq k-1$. Clearly, there are $r^{\lfloor \frac{n+1}{i}\rfloor}$ tilings without $k$-rectangles. We have $n+1=l+qi=l-k+k+i+(q-1)i$, for some integer $q\geq 2$. If $l\leq 2k-i-1$, as in the previous case, there are $sr^{\lfloor \frac{n+1-k}{i}\rfloor}(\lfloor \frac{n+1-k}{i}\rfloor+1)$ tilings having exactly one $k$-rectangle. On the other hand, when $2k-i-1<l\leq k-1$ there do not exist tilings with exactly one $k$-rectangles. When $i>k$ and $l>k-1$, all tilings have a $k$-rectangle. Then, if $l\leq 2k-1$, we have $n+1=l-k+k+qi$, thus there are exactly $sr^{\lfloor \frac{n+1-k}{i}\rfloor}(\lfloor \frac{n+1-k}{i}\rfloor+1)$ tilings having only one $k$-rectangle, but when $l>2k-1$ there do not exist such tilings. Hence $u(i,k,n)$ enumerates the tilings having at most one $k$-rectangle. 
	
	Now, we count the number of the remaining tilings enumerated by $F_{r,s}^{i}(k,n)$ considering the number of $i$-rectangles between the rightmost two $k$-rectangles. There are   $s^{2}r^{j+t}F_{r,s}^{i}(k,n-2k-ti-ji)$ tilings ending with $j$ $i$-rectangles and $t$ $i$-rectangles between the last two $k$-rectangles, $t \in \{0, \ldots, \lfloor \frac{n+1-2k-ji}{i}\rfloor \}$. Adding up all these numbers we get the total of tilings counted by $F_{r,s}^{i}(k,n)$.	
\end{proof}

The identity below, see \textnormal{\cite{robson}}, follows from \eqref{31.3}:
$$F_{n}=(n-1)+\sum\limits_{j=1}^{n-4}(n-3-j)F_{j}, \forall n\geq 5.$$

\begin{corollary}
	\textit{}
	\begin{enumerate}
		\setlength\itemsep{-0.05cm}
		\item $F_{2n}=n+\sum\limits_{j=1}^{n-1}(n-j)F_{2j}, \forall n\geq 2$,
		\item $F_{2n+1}=1+\sum\limits_{j=1}^{n}(n+1-j)F_{2j-1}, \forall n\geq 1$,
		\item $P_{n}=2^{n-1}+2^{n-3}(n-2)+\sum\limits_{j=1}^{n-4}2^{n-4-j}(n-3-j)P_{j}, \forall n\geq 5$,
		\item $P_{2n}=2n+4\sum\limits_{j=1}^{n-1}(n-j)P_{2j}, \forall n\geq 2$,
		\item $P_{2n+1}=1+4\sum\limits_{j=1}^{n}(n+1-j)P_{2j-1}, \forall n\geq 1$,
		\item $Pv(2n)=n+\sum\limits_{j=0}^{n-3}(n-2-j)Pv(2j), \forall n\geq 3$,
		\item $Pv(2n+1)=(n+1)+\sum\limits_{j=0}^{n-3}(n-2-j)Pv(2j+1), \forall n\geq 3$,
		\item $Pv(3n)=1+n+\sum\limits_{j=1}^{n-1}(n-j)Pv(3j-1), \forall n\geq 2$,
		\item $Pv(3n+1)=1+\sum\limits_{j=0}^{n-1}(n-j)Pv(3j), \forall n\geq 1$,
		\item $Pv(3n+2)=1+n+\sum\limits_{j=1}^{n}(n+1-j)Pv(3j-2), \forall n\geq 1$,
		\item $J_{n}=2n-3+4\sum\limits_{j=1}^{n-4}(n-3-j)J_{j}, \forall n\geq 5$,
		\item $J_{2n}=n2^{n-1}+\sum\limits_{j=1}^{n-1}2^{n-1-j}(n-j)J_{2j}, \forall n\geq 2$,
		\item $J_{2n+1}=2^{n}+\sum\limits_{j=1}^{n}2^{n-j}(n+1-j)J_{2j-1}, \forall n\geq 1$,
		\item $u_{n}=n-2+\sum\limits_{j=1}^{n-6}(n-5-j)u_{j}, \forall n\geq 7$,
		\item $u_{3n}=\sum\limits_{j=1}^{n}(n+1-j)u_{3j-2}, \forall n\geq 1$,
		\item $u_{3n+1}=1+\sum\limits_{j=1}^{n}(n+1-j)u_{3j-1}, \forall n\geq 1$,
		\item $u_{3n+2}=n+1+\sum\limits_{j=1}^{n}(n+1-j)u_{3j}, \forall n\geq 1$,
		\item $\{n\}_{s,t}=s^{n-1}+nts^{n-3}+t^{2}\sum\limits_{j=1}^{n-4}s^{n-4-j}(n-3-j)\{j\}_{s,t}, \forall n\geq 5$,
		\item $\{2n\}_{s,t}=nts^{n-1}+t^{2}\sum\limits_{j=1}^{n-1}s^{n-j-1}(n-j)\{2j\}_{s,t}, \forall n\geq 2$, 
		\item $\{2n+1\}_{s,t}=s^{n}+t^{2}\sum\limits_{j=1}^{n}s^{n-j}(n+1-j)\{2j-1\}_{s,t}, \forall n\geq 1$.
	\end{enumerate}
	\label{corolid}
\end{corollary}

Some more interesting identities can be derived when we combine Corollary \ref{corolid} with the identities previously obtained:

\begin{corollary}
	\textit{}
	\begin{enumerate}
		\setlength\itemsep{-0.05cm}
		\item $\sum\limits_{j=0}^{\lfloor \frac{n-1}{2}\rfloor}{n-1-j \choose j}=1+\sum\limits_{j=1}^{n-2}F_{j}=(n-1)+\sum\limits_{j=1}^{n-4}(n-3-j)F_{j}, \forall n\geq 5$,
		\item $\sum\limits_{j=0}^{n-1}F_{2j+1}=n+\sum\limits_{j=1}^{n-1}(n-j)F_{2j}, \forall n\geq 2$,
		\item $\sum\limits_{j=0}^{n}F_{2j}=\sum\limits_{j=1}^{n}(n+1-j)F_{2j-1}, \forall n\geq 1$,
		\item $\sum\limits_{j=0}^{\lfloor \frac{n-1}{2}\rfloor}{n-1-j \choose j}2^{n-1-2j}=2^{n-3}(3-n)+2^{n-2}\sum\limits_{j=2}^{n-2}2^{-j}P_{j}=\sum\limits_{j=1}^{n-4}2^{n-4-j}(n-3-j)P_{j}, \forall n\geq 5$,
		\item $\sum\limits_{j=0}^{n-1}P_{2j+1}=n+2\sum\limits_{j=1}^{n-1}(n-j)P_{2j}, \forall n\geq 2$,
		\item $\sum\limits_{j=0}^{n}P_{2j}=2\sum\limits_{j=1}^{n}(n+1-j)P_{2j-1}, \forall n\geq 1$,
		\item $1+\sum\limits_{j=0}^{n-2}Pv(2j+1)=n+\sum\limits_{j=0}^{n-3}(n-2-j)Pv(2j), \forall n\geq 3$,
		\item $\sum\limits_{j=0}^{n-1}Pv(2j)=n+\sum\limits_{j=0}^{n-3}(n-2-j)Pv(2j+1), \forall n\geq 3$,
		\item $\sum\limits_{j=0}^{n-1}Pv(3j+1)=n+\sum\limits_{j=1}^{n-1}(n-j)Pv(3j-1), \forall n\geq 2$,
		\item $\sum\limits_{j=0}^{n-1}Pv(3j+2)=\sum\limits_{j=0}^{n-1}(n-j)Pv(3j), \forall n\geq 1$,
		\item $\sum\limits_{j=1}^{n}Pv(3j)=n+\sum\limits_{j=1}^{n}(n+1-j)Pv(3j-2), \forall n\geq 1$,
		\item $3+2\sum\limits_{j=2}^{n-2}J_{j}=\sum\limits_{j=0}^{\lfloor \frac{n-1}{2}\rfloor}{n-1-j\choose j}2^{j}=2n-3+4\sum\limits_{j=1}^{n-4}(n-3-j)J_{j}, \forall n\geq 5$,
		\item $\sum\limits_{j=0}^{n-1}2^{n-1-j}J_{2j+1}=n2^{n-1}+\sum\limits_{j=1}^{n-1}2^{n-1-j}(n-j)J_{2j}, \forall n\geq 2$,
		\item $\sum\limits_{j=1}^{n}2^{-j}J_{2j}=\sum\limits_{j=1}^{n}2^{-j}(n+1-j)J_{2j-1}, \forall n\geq 1$,
		\item $\sum \limits _{j=0}^{\lfloor \frac{n}{3}\rfloor} \dbinom{n-2j}{j}=n-2+\sum\limits_{j=1}^{n-6}(n-5-j)u_{j}, \forall n\geq 7$,
		\item $\sum \limits_{j=0}^{n-1}u_{3j+2}=\sum\limits_{j=1}^{n}(n+1-j)u_{3j-2}, \forall n\geq 1$,
		\item $\sum \limits_{j=1}^{n}u_{3j}=\sum\limits_{j=1}^{n}(n+1-j)u_{3j-1}, \forall n\geq 1$,
		\item $\sum\limits_{j=0}^{n}u_{3j+1}=n+1+\sum\limits_{j=1}^{n}(n+1-j)u_{3j}, \forall n\geq 1$,
		\item $\sum \limits _{k\geq 0} \dbinom{n\!-\!k\!-\!1}{k}s^{n-2k-1}t^k=s^{n-1}+nts^{n-3}+t^{2}\sum\limits_{j=1}^{n-4}s^{n-4-j}(n-3-j)\{j\}_{s,t}, \forall n\geq 5$,
		\item $st^{n-1}+\sum\limits_{j=0}^{n-2}st^{n-2-j}\{2j+3\}_{s,t}=nts^{n-1}+t^{2}\sum\limits_{j=1}^{n-1}s^{n-j-1}(n-j)\{2j\}_{s,t}, \forall n\geq 2$, 
		\item$t^{n}+\sum\limits_{j=0}^{n-1}st^{n-1-j}\{2j+2\}_{s,t}=s^{n}+t^{2}\sum\limits_{j=1}^{n}s^{n-j}(n+1-j)\{2j-1\}_{s,t}, \forall n\geq 1$.
	\end{enumerate}
\end{corollary}


\begin{theorem} Let $i,k\geq 1 $ and $n\geq 2\max \lbrace i,k \rbrace-1$ be integers. Let $l$ be the remainder of the division of $n+1$ by $i$. Then, 
	\begin{equation}
	\begin{array}{lcl}
	F_{r,s}^{i}(k,n) & = & v(i,k,n)+ \sum\limits_{j=0}^{\lfloor \frac{n+1-2k}{i}\rfloor}r^{j}s^{2}F_{r,s}^{i}(k,n-2k-ji) \\ & & +\sum\limits_{j=0}^{\lfloor \frac{n+1-k-i}{i}\rfloor}r^{j+1}sF_{r,s}^{i}(k,n-i-k-ji), \end{array}
	\label{31.4}
	\end{equation}
	where
	\begin{equation*}
	v(i,k,n)=
	\left \{ 
	\begin{array}{ll}
	r^{\lfloor \frac{n+1}{i}\rfloor}+sr^{\lfloor \frac{n+1-k}{i}\rfloor}, &  \textnormal{\textit{if}} \ k\geq i\ \textnormal{\textit{or}} \ i>k \ \textnormal{\textit{and}} \ l\leq 2k-i-1\\
	r^{\lfloor \frac{n+1}{i}\rfloor}, & \textnormal{\textit{if}} \ i>k \ \textnormal{\textit{and}} \ 2k-i-1<l\leq k-1 \\
	sr^{\lfloor \frac{n+1-k}{i}\rfloor}, & \textnormal{\textit{if}} \ i>k\ \textnormal{\textit{and}} \ k\leq l\leq 2k-1 \\
	{0}, & {\textnormal{\textit{if}} \ i>k \ \textnormal{\textit{and}} \ l>2k-1.}\\
	\end{array}
	\right.
	\end{equation*}
\end{theorem}

\begin{proof} 
	
	Firstly, we show that $v(i,k,n)$ is the number of type $\mathcal{F}$ tilings without $k$-rectangles before the last piece. When $k\geq i$, it is easy to see that there are $r^{\lfloor \frac{n+1}{i}\rfloor}+sr^{\lfloor \frac{n+1-k}{i,}\rfloor}$ tilings of an $(n+1)$-board that do not have any $k$-rectangle before the last piece, $r^{\lfloor \frac{n+1}{i}\rfloor}$ if the last piece is an $i$-rectangle, and $sr^{\lfloor \frac{n+1-k}{i}\rfloor}$ if the last piece is a $k$-rectangle. If $i>k$ and $l\leq 2k-i-1<k-1$, then we have $n+1=l+qi=l-k+k+i+(q-1)i$, for some integer $q\geq 2$. Thus, there exist tilings without $k$-rectangles before the last piece. On the other hand, when $2k-i-1<l\leq k-1$, such tilings end only with $i$-rectangles. Now, suppose $i>k$ and $k\leq l\leq i-1$. In this case, there are not tilings with $k$-rectangles before the last piece, where the last piece is an $i$-rectangle. We also have $n+1=l+qi=l-k+k+qi$. Hence, if $l\leq 2k-1$, there exist tilings, whose last piece is a $k$-rectangle, without $k$-rectangles before the last piece, and if $l > 2k-1$ we do not have any tiling without $k$-rectangles before the last piece. In the remaining cases there is a $k$-rectangle followed by $j$ $i$-rectangles before the last piece, where either $0\leq j \leq \lfloor\frac{n+1-2k}{i}\rfloor$, if the last piece is a $k$-rectangle or $0\leq j \leq \lfloor\frac{n+1-k-i}{i}\rfloor$ otherwise. The number of tilings in the former is $s^{2}r^{j}F_{r,s}^{i}(k,n-2k-ji)$ while in the last case we have $sr^{j+1}F_{r,s}^{i}(k,n-i-k-ji)$ tilings. By adding up all these numbers we conclude the proof.
\end{proof}


Before presenting our next identity, we need to define the idea of a \textit{breakable} tiling. Similar objects have been defined by other authors, see \cite[Section 1.2]{benjamin2003} or \cite[Section 2.1]{benjamin2008tiling} for instance.

Let $k>i\geq 1$. We say that a type $\mathcal{F}$ tiling of an $(n+1)$-board is \textit{breakable} at cell $m$ if we can decompose it into two tilings, one covering cells $1$ until $m$ and the other covering the cells from $m+1$ until $n+1$. We call  a tiling \textit{unbreakable} at cell $m$, if a $k$-rectangle ($i$-rectangle if $i>1$) occupies cells $m-j, \ldots, m+k-1-j$ ($m+i-1-j$), where $j\in \lbrace 0, \ldots, k-2\rbrace$.   

\begin{example} Figure \ref{fig4} below shows a tiling breakable at cells 1, 3, 5, and 8.
\end{example}
\begin{figure}[h!]
	\centering
	\begin{tikzpicture}[scale=0.8][rounded corners, ultra thick]
	\shade[top color=black,bottom color=black, draw=black] (0,-2) rectangle +(0.5,0.5);	
	\shade[top color=white,bottom color=white, draw=black] (0.5,-2) rectangle +(1.5,0.5);
	\shade[top color=white,bottom color=white, draw=black] (1.5,-2) rectangle +(1.0,0.5);
	\shade[top color=gray,bottom color=gray, draw=black] (2.5,-2) rectangle +(1.5,0.5);
	\end{tikzpicture}
	\caption{A breakable tiling}
	\label{fig4}	
\end{figure}
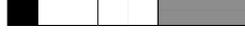


\begin{theorem} Let $m \geq k\geq 2$ and $n\geq 0$ be integers. Then,
	\begin{equation}
	F^{1}_{r,s}(k,m+n)=F^{1}_{r,s}(k,m-1)F^{1}_{r,s}(k,n)+s\sum_{j=0}^{k-2}F^{1}_{r,s}(k,m-j-2)F^{1}_{r,s}(k,n-k+j+1).
	\label{quebr}
	\end{equation}
	\label{quebravel}
\end{theorem}

\begin{proof} We show that the r.h.s. of \eqref{quebr} is also the number of tilings enumerated by $F^{1}_{r,s}(k,m+n)$ by considering whether or not a tiling is breakable at cell $m$.  It is easily seen that if a tiling of an $(m+n+1)$-board is breakable at cell $m$, then it splits into two tilings, one of length $m$ and another one of length $n+1$. Hence we have $F^{1}_{r,s}(k,m-1)F^{1}_{r,s}(k,n)$ breakable tilings at cell $m$.
	
	If a tiling of an $(m+n+1)$-board is unbreakable at cell $m$, then there exists a $k$-rectangle covering the cells $m-j, \ldots, m+k-1-j$, for some $j\in \lbrace 0, \ldots, k-2\rbrace$. In this case, the tiling can be split into a tiling of length $m-j-1$ followed by one $k$-rectangle, which in turn is followed by a tiling of length $n-k+j+2$. As we have $s$ colors to the $k$-rectangles, there are  $s\sum_{j=0}^{k-2}F^{1}_{r,s}(k,m-j-2)F^{1}_{r,s}(k,n-k+j+1)$ tilings that are unbreakable at cell $m$.
	Summing over all possible tilings we finish the proof.
\end{proof}

It follows from this theorem the known identities $F_{m+n}=F_{m+1}F_{n}+F_{m}F_{n-1}$ (\cite{benjamin2008tiling}), $P_{m+n}=P_{m+1}P_{n}+P_{m}P_{n-1}$ (\cite{koshpell}), $u_{m+n}=u_{m+1}u_{n}+u_{m}u_{n-2}+u_{m-1}u_{n-1}$ (\cite{flaut2013generalized}), $\{m+n\}_{s,t}=\{n+1\}_{s,t}\{m\}_{s,t}+t\{n\}_{s,t}\{m-1\}_{s,t}$ (\cite{amdeberhan2014generalized}), and also an analogous identity for the Jacobsthal numbers:

\begin{corollary} For $n, m \geq 1$, we have
	$$J_{m+n}=J_{m+1}J_{n}+2J_{m}J_{n-1}.$$
	\label{13.1}
\end{corollary}



We call \textit{tail} of a type $\mathcal{L}$ tiling either the last sequence of $k$ pieces, if there is a $k$-rectangle after all or between two $i$-rectangles (tail of size $k(1+i)-i = (k-1)i+k$), or the last sequence of $k-1$ pieces, if the tiling ends with $k-1$ $i$-rectangles (tail of size $(k-1)i$).

\begin{example} In Figure \ref{fig5} we have the possible tails of the tilings of a $9$-board using $2$-rectangles and  $3$-rectangles, of colors white and gray, respectively.	
	\begin{figure}[h!]
		\centering
		\begin{tikzpicture}[scale=0.8][rounded corners, ultra thick]
		\draw (-0.25,-0.75) ;
		\shade[top color=white,bottom color=white, draw=black] (0,-1) rectangle +(1.0,0.5);
		\shade[top color=white,bottom color=white, draw=black] (1.0,-1) rectangle +(1.0,0.5);    
		\draw (-0.25,-0.75) ;
		\shade[top color=white,bottom color=white, draw=black] (0,-2) rectangle +(1.0,0.5);    
		\shade[top color=gray,bottom color=gray, draw=black] (1.0,-2) rectangle +(1.5,0.5);    
		\shade[top color=white,bottom color=white, draw=black] (2.5,-2) rectangle +(1.0,0.5);    
		\draw (-0.25,-0.75) ;
		\shade[top color=white,bottom color=white, draw=black] (0,-3) rectangle +(1.0,0.5);    
		\shade[top color=white,bottom color=white, draw=black] (1.0,-3) rectangle +(1.0,0.5);
		\shade[top color=gray,bottom color=gray, draw=black] (2.0,-3) rectangle +(1.5,0.5);
		\end{tikzpicture}
		\caption{Tails of a type $\mathcal{L}$ tiling using $2$-rectangles and $3$-rectangles}
		\label{fig5}	
	\end{figure}
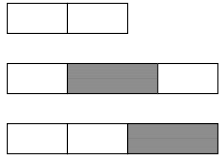
	

\end{example}

\begin{theorem} Let $k > i$ and $n\geq k(1+i)-i-1$ be integers. Then, 
	$$L_{r,s}^{i}(k,n)=ksr^{k-1}F_{r,s}^{i}(k,n-(1+i)k+i)+r^{k}F_{r,s}^{i}(k,n-ki).$$
\end{theorem}

\begin{proof} We know that $(k-1)sr^{k-1}F_{r,s}^{i}(k,n-k(i+1)+i)$ equals the number of type $\mathcal{L}$ tilings with tail of size $k(i+1)-i$, while $r^{k-1}F_{r,s}^{i}(k,n-(k-1)i)$ is the number of type $\mathcal{L}$ tilings with tail of size $(k-1)i$. We decompose the last set of tilings into two disjoint subsets according to the last piece being either an $i$-rectangle, whose total is $r^{k}F_{r,s}^{i}(k,n-(k-1)i-i)=r^{k}F_{r,s}^{i}(k,n-ki)$, or a $k$-rectangle, whose total equals $sr^{k-1}F^{i}_{r,s}(k,n-(k-1)i-k)=sr^{k-1}F^{i}_{r,s}(k,n-(i+1)k+i)$. Then, adding up these numbers we are left with the number of tilings enumerated by $L_{r,s}^{i}(k,n)$.
\end{proof} 

Taking $i=r=s=1$ and $k=2$ in the above identity it follows the known result: $L_n=F_{n}+2F_{n-1}.$

\begin{theorem} Let $k > i$ and $n\geq k(2+i)-i-1$ be integers. Then,
	$$\begin{array}{rl} L_{r,s}^{i}(k,n)= & (k-1)sr^{\lfloor \frac{n+1-k}{i} \rfloor} +r^{\lfloor \frac{n+1}{i} \rfloor} \\ & +\sum\limits_{j=0}^{\lfloor \frac{n+1-k(2+i)+i}{i}\rfloor}(k-1)s^2r^{(k-1)+j}F_{r,s}^{i}(k,n-k(2+i)+i-ij) \\ & +\sum\limits_{t=0}^{\lfloor \frac{n+1-k(1+i)+i}{i}\rfloor}sr^{(k-1)+t}F_{r,s}^{i}(k,n-k(1+i)+i-it). \end{array}$$
\end{theorem}

\begin{proof} It is easily seen that $(k-1)sr^{\lfloor \frac{n+1-k}{i} \rfloor} +r^{\lfloor \frac{n+1}{i} \rfloor}$ is the number of type $\mathcal{L}$ tilings of an $(n+1)$-board not having any $k$-rectangle before the tail: $(k-1)sr^{\lfloor \frac{n+1-k}{i} \rfloor}$ having tail of size $k(1+i)-i$ and $r^{\lfloor \frac{n+1}{i} \rfloor}$ of size $(k-1)i$. In the remaining cases there exists at least one $k$-rectangle before the tail. We count the number of such tilings by considering the number, $j$, of $i$-rectangles between the tail and the rightmost $k$-rectangle before the tail. We have $j\leq \lfloor \frac{n+1-(2k+(k-1)i)}{i}\rfloor=\lfloor \frac{n+1-k(2+i)+i}{i}\rfloor$, if the tail has size $k(1+i)-i$ and $t\leq \lfloor \frac{n+1-(k+(k-1)i)}{i}\rfloor=\lfloor \frac{n+1-k(1+i)+i}{i}\rfloor$, if the tail is of size $(k-1)i$. Removing this sequence of pieces we are left with either type $\mathcal{F}$ tilings of an $(n+1-k(2+i)+i-ij)$-board, if the tail has size $k(1+i)-i$ or type $\mathcal{F}$ tilings of an $(n+1-k(1+i)+i-it)$-board, if the tail is of size $(k-1)i$. Summing all possible values in each case we finish the proof, observing that the factor $(k-1)$ in the first sum stems from the number of possibilities to insert a $k$-rectangle in a tail of size $k(1+i)-i$.
\end{proof}

\begin{corollary} For $n\geq 4$, we have
	$$L_n= 2 + F_n+2\sum \limits _{j=1}^{n-3}F_{j}$$
	and consequently $L_n - F_n$ is even.
	\label{13.1.nova}
\end{corollary}


\begin{theorem} Let $k > i $ and $n\geq k(3+i)-i-1$ be integers. Then,
	{\small
		$$\begin{array}{rcl}
		L^{i}_{r,s}(k,n)\! & = & (k\!-\!1)sr^{\lfloor \frac{n\!+\!1\!-\!k}{i} \rfloor} +r^{\lfloor \frac{n\!+\!1}{i} \rfloor}\!+\!(k\!-\!1)s^{2}r^{\lfloor \frac{n\!+\!1\!-2k}{i}\rfloor}\left( \lfloor \frac{n+1-k(2+i)+i}{i}\rfloor \!+\! 1\right) \\ & & + sr^{\lfloor \frac{n\!+\!1\!-\!k}{i}\rfloor}\left( \lfloor \frac{n\!+\!1\!-\!k(1\!+\!i)\!+\!i}{i}\rfloor \!+\! 1\right) \\ & & 
		+\sum\limits_{j=0}^{\lfloor \frac{n\!+\!1\!-\!k(3\!+\!i)\!+\!i}{i}\rfloor}\sum\limits_{t=0}^{\lfloor \frac{n\!+\!1\!-k(3\!+\!i)\!+\!i\!-\!ji}{i}\rfloor}(k\!-\!1)s^3r^{(k\!-\!1)\!+\!j\!+\!t\!}F_{r,s}^{i}(k,n\!-\!k(3\!+\!i)\!+\!i\!-\!ji\!-\!ti) \\ & & 
		+\sum\limits_{j=0}^{\lfloor\frac{n\!+\!1\!-\!k(i\!+\!2)\!+\!i}{i}\rfloor}\sum\limits_{t=0}^{\lfloor\frac{n\!+\!1\!-\!k(i\!+\!2)\!+\!i\!-ji}{i}\rfloor}s^2r^{(k\!-\!1)\!+\!j\!+\!t}F_{r,s}^{i}(k,n\!-\!k(i\!+\!2)\!+\!i\!-ji\!-\!ti).
		\end{array}$$}
\end{theorem}

\begin{proof} We count the number of type $\mathcal{L}$ tilings of an $(n+1)$-board considering the number of $k$-rectangles before the tail. There are $(k-1)sr^{\lfloor \frac{n+1-k}{i} \rfloor} +r^{\lfloor \frac{n+1}{i} \rfloor}$ tilings not having any $i$-rectangle before the tail. Now, we consider the tilings with exactly one $k$-rectangle before the tail. If the tail has size $k(1+i)-i$, we have  $(k-1)s^{2}r^{\lfloor \frac{n+1-2k}{i}\rfloor}\left( \lfloor \frac{n+1-k(2+i)+i}{i}\rfloor +1\right)$ tilings. On the other hand, if the size of the tail is $(k-1)i$, then  there are $sr^{\lfloor \frac{n+1-k}{i}\rfloor}\left( \lfloor \frac{n+1-k(1+i)+i}{i}\rfloor +1\right)$ tilings. In both cases there exist at least two $k$-rectangles before the tail. Denoting by $j$ the number of $i$-rectangles between the last two $k$-rectangles before the tail and by $t$ the number of $i$-rectangles between the rightmost $k$-rectangle and the tail, we have two cases to consider: \\
	\noindent \textit{(a) tail of size $k(1+i)-i$}. In this case, $j, t\leq \lfloor \frac{n+1-k(1+i)+i-2k}{i}\rfloor=\lfloor \frac{n+1-k(3+i)+i}{i}\rfloor$ and for each $j$ and $t$, there are $(k-1)s^{3}r^{(k-1)+j+t} F_{r,s}^{i}(k,n-k(3+i)+i-ji-ti)$ tilings. Hence,  $$\sum\limits_{j=0}^{\lfloor \frac{n+1-k(3+i)+i}{i}\rfloor}\sum\limits_{t=0}^{\lfloor \frac{n+1-k(3+i)+i-ji}{i}\rfloor}(k-1)s^3r^{(k-1)+j+t}F_{r,s}^{i}(k,n-k(3+i)+i-ji-ti)$$
	enumerates the total of type $\mathcal{L}$ tilings of an $(n+1)$-board with tail of size $k(1+i)-i$ and a sequence of $j$ $i$-rectangles between the last two $k$-rectangles followed by $t$ $i$-rectangles before the tail. \\
	\noindent \textit{(b) tail of size $(k-1)i$}. In analogy with the previous case, we have $j,t\leq \lfloor \frac{n+1-k(i+2)+i}{i}\rfloor$ and, then, \\
	$$\sum\limits_{j=0}^{\lfloor\frac{n+1-k(i+2)+i}{i}\rfloor}\sum\limits_{t=0}^{\lfloor\frac{n+1-k(i+2)+i-ji}{i}\rfloor}s^2r^{(k-1)+j+t}F_{r,s}^{i}(k,n-k(i+2)+i-ji-ti)$$ 
	is the number of tilings.
\end{proof}

\begin{corollary} If $n\geq 6$, we have
	$$L_n=2(n-1)+F_{n-3}+\sum_{j=0}^{n-5} 2(j+1)F_{n-4-j}$$
	and, consequently, $L_n - F_{n-3}$ is even.
	\label{Cor19.2}
\end{corollary}


Combining Corollaries \ref{13.1.nova} and \ref{Cor19.2}, if follows that $F_n - F_{n-3}$ is even for $n \geq 6$.

\begin{theorem} Let $k >i\geq 2$ and $n\geq 2k+(k-1)i-1$ be integers such that $n+1-(k-1)i \equiv l (mod\  k)$, where $0 \leq l \leq i-1$. Then,
	$$\begin{array}{rl} L^{i}_{r,s}(k,n) = & s^{\lfloor \frac{n+1-(k-1)i}{k}\rfloor}+(k-1)s^{\lfloor \frac{n+1+i-k(i+1)}{k}\rfloor} \\ & + r^ks^{j}\sum\limits_{j=0}^{\lfloor \frac{n+1-(k-1)i}{k} \rfloor}F^{i}_{r,s}(k,n-ki-jk) \\ & + (k-1)r^ks^{j+1}\sum\limits_{j=0}^{\lfloor \frac{n+1-(k-1)i-k}{k} \rfloor}F^{i}_{r,s}(k,n-ik-jk-k). \end{array}$$ 
\end{theorem}

\begin{proof} There are $s^{\lfloor \frac{n+1-(k-1)i}{k}\rfloor}$ and $(k-1)s^{\lfloor \frac{n+1+i-k(i+1)}{k}\rfloor}$ type $\mathcal{L}$ tilings of an $(n+1)$-board without $i$-rectangles before the tail for tail of size $(k-1)i$ and $(k-1)i+k$, respectively. The other tilings end with an  $i$-rectangle followed by a number $j$ of $k$-rectangles before the tail. Thus, by removing this sequence of pieces and the tail, we are left with type $\mathcal{F}$ tilings of either an $(n+1-(j+i)k)$-board or an $(n+1-(j+i+1)k)$-board. In the latter case, we have $r^ks^{j}F^{i}_{r,s}(k,n-ki-jk)$ tilings for each $j$, while in the former case we have $(k-1)r^ks^{j+1}F^{i}_{r,s}(k,n-ki-jk-k)$. Adding up these numbers we conclude the proof.
\end{proof}

\begin{theorem} Let $k>i= 1$ and $n\geq 3k-2$ be integers. then, 
	$$\begin{array}{rl} L^{i}_{r,s}(k,n) = & u(n) + r^k s^{j}\sum_{j=0}^{\lfloor \frac{n+1-k}{k} \rfloor}F^{i}_{r,s}(k,n-jk-k) \\ & + r^k s^{j+1}(k-1)\sum_{j=0}^{\lfloor \frac{n+1-2k}{k} \rfloor}F^{i}_{r,s}(k,n-jk-2k), \end{array}$$
	where
	\begin{equation*}
	u(n)=
	\left \{
	\begin{array}{ll}
	0, & \textnormal{if} \ n+2-k \not\equiv 0 \ (mod \ k) \\
	r^{k-1}s^{\lfloor \frac{n+2-k}{k}\rfloor}+(k-1)r^{k-1}s^{\lfloor \frac{n+2-2k}{k}\rfloor}, & \textnormal{if} \   n+2-k \equiv 0 \ (mod \ k). \\
	\end{array}
	\right.
	\end{equation*}
\end{theorem}

The proof of this theorem will be omitted since it is analogous to the previous one. We shaw observe that when $n+2-k \not\equiv 0 (mod \ k)$ there do not exist tilings without $i$-rectangles before the tail. As a consequence of this theorem we have:

\begin{corollary} For $n\geq4$, we have
	\begin{equation*}
	L_n=
	\left \{
	\begin{array}{ll}
	2+F_n+2\displaystyle\sum_{j=0}^{\frac{n}{2} -2}F_{n-2-2j}, & \textnormal{if $n$ is even}, \\
	F_n+2\displaystyle\sum_{j=0}^{\frac{n-3}{2}}F_{n-2-2j}, & \textnormal{if $n$ is odd}.  
	\end{array}
	\right.
	\end{equation*}
\end{corollary}


\section{Matrix generators}
\label{Sec6}

In this section, we discuss an important tool, the matrix methods, for obtaining results about recurrence relations. Some examples of the power of this method can be seen, among others, in \cite{brualdi1991matrix, distance2}.

Let 
$$Q=\left[ \begin{array}{cc}
1 & 1 \\
1 & 0
\end{array} \right] \mbox{ and } P=\left[ \begin{array}{cc}
2 & 1 \\
1 & 0
\end{array} \right].$$
Taking $n$-th power, we get the Fibonacci and Pell numbers,
\begin{center}$Q^{n}=\left[ \begin{array}{cc}
	F_{n+1} & F_{n} \\
	F_{n} & F_{n-1}
	\end{array} \right]$ 
	\mbox{ and }
	$P^{n}=\left[ \begin{array}{cc}
	P_{n+1} & P_{n} \\
	P_{n} & P_{n-1}
	\end{array} \right],$
\end{center}
from where we have the Cassini formulas: $\det Q^{n}=(-1)^{n}=F_{n+1}F_{n-1}-F_{n}^{2}$ and $\det P^{n}=(-1)^n=P_{n+1}P_{n-1}-P_{n}^{2}$.

Let $k\geq i$. Define $Q_{k}=[q_{tj}]_{k\times k}$, where, for each given $1\leq t\leq k$, $q_{t1}$ equals the coefficient of $F^{i}_{r,s}(k,n-t)$ in the recurrence \eqref{23.1} defining $F^{i}_{r,s}(k,n)$ and, for $j\geq 2$,  
\begin{equation*}
q_{tj}=
\left \{
\begin{array}{cc}
1, & \textnormal{if} \ j=t+1 \\
0, & \textnormal{otherwise}.  \\
\end{array}
\right.
\end{equation*}

For example, for $i=2$ and $k=2,3, \ldots, k$ we have:
$$Q_2=\left[ \begin{array}{cc}
0 & 1 \\
r+s & 0
\end{array} \right], Q_3=\left[ \begin{array}{ccc}
0 & 1  & 0 \\
r & 0 & 1\\
s & 0  & 0
\end{array} \right], Q_4=\left[ \begin{array}{cccc}
0 & 1  & 0 & 0 \\
r & 0 & 1 & 0\\
0 & 0  & 0& 1\\
s & 0  & 0&  0\\
\end{array} \right],\ldots,$$

$$Q_k=\left[ \begin{array}{ccccc}
0 & 1  & 0 & \ldots & 0\\
r & 0 & 1 & \ldots& 0\\
\vdots & \vdots  & \vdots& \ddots & \vdots \\
0 & 0  & 0&  \ldots& 1\\
s & 0  & 0&  \ldots & 0
\end{array} \right].$$\\

We will also need the following matrix of initial conditions:
$$A_k=\left[ \begin{array}{cccc}
F^{i}_{r,s}(k,2k-3) & F^{i}_{r,s}(k,2k-4)  & \ldots & F^{i}_{r,s}(k,k-2)\\
F^{i}_{r,s}(k,2k-4) & F^{i}_{r,s}(k,2k-5)  & \ldots& F^{i}_{r,s}(k,k-3)\\
\vdots & \vdots  & \ddots & \vdots \\
F^{i}_{r,s}(k,k-1) & F^{i}_{r,s}(k,k-2)  &  \ldots & F^{i}_{r,s}(k,0)\\
F^{i}_{r,s}(k,k-2) & F^{i}_{r,s}(k,k-3)  &  \ldots & F^{i}_{r,s}(k,-1)
\end{array} \right].$$

Now, we are ready to present the main result of this section. The proof is easily get by induction. 

\begin{theorem} Let $k\geq 2$ and $n\geq -1$ be integers. Then, 
	$$A_kQ_{k}^{n+1}=\left[ \begin{array}{ccc}
	F^{i}_{r,s}(k,n+2k-2)   & \ldots & F^{i}_{r,s}(k,n+k-1)\\
	F^{i}_{r,s}(k,n+2k-3)   & \ldots & F^{i}_{r,s}(k,n+k-2)\\
	\vdots   & \ddots & \vdots \\
	F^{i}_{r,s}(k,n+k)   &  \ldots & F^{i}_{r,s}(k,n+1)\\
	F^{i}_{r,s}(k,n+k-1)   &  \ldots & F^{i}_{r,s}(k,n)
	\end{array} \right].$$
\end{theorem}

In order to get the generalization of the Cassini formula, we will need some results.

\begin{theorem} Let $k\geq 2$. Then, 
	\begin{equation*}
	\det Q_k=
	\left \{
	\begin{array}{ll}
	s(-1)^{k+1}, & \textnormal{if} \ i<k \\
	(s+r)(-1)^{k+1}, & \textnormal{if} \ i=k.  \\
	\end{array}
	\right.
	\end{equation*}
\end{theorem}

\begin{proof} The determinant can be easily calculated when we use the cofactor formula on the fist column of $Q_k$.
\end{proof}

\begin{theorem} Let $k\geq 2$ and $i=1$. Then, 
	$$\det A_k=s^{k-1}(-1)^{\frac{(k-1)(k+6)}{2}}.$$
\end{theorem}

\begin{corollary} Let $k\geq 2$, $i=1$, and $n\geq 0$. Then,
	\begin{equation*}
	\det A_{k}Q_k^{n+1}=s^{n+k}(-1)^{\frac{(k-1)(k+6)+2(k+1)(n+1)}{2}}.
	\end{equation*}
	\label{corocassini}	
\end{corollary}

\section*{Acknowledgments}

The first author thanks the CNPq for the financial support. The second and third authors are thankful to FAPEAM for the financial support.



\noindent {Departamento de Ci\^encia e Tecnologia, Universidade Federal de S\~ao Paulo - UNIFESP, S\~ao Jos\'e dos Campos - SP, 12247-014, Brazil, silva.robson@unifesp.br}

\

\noindent {Universidade Federal do Amazonas, Manaus - AM, 69103-128, Brazil, ksouzamath@yahoo.com.br}

\

\noindent {Universidade do Estado do Amazonas, Manaus - AM, 69055-038, Brazil, agneto@uea.edu.br}

\end{document}